\documentclass[12pt]{article}
\usepackage[utf8]{inputenc}
\usepackage{comment}
\usepackage[english]{babel}
\usepackage{amsmath,amsfonts,amssymb,amsthm}
\usepackage{mathdots}
\usepackage{amscd}
\usepackage[a4paper,nomarginpar]{geometry}
\usepackage{epsfig,graphicx,color}
\usepackage[colorlinks=true, linkcolor=blue, urlcolor=red, citecolor=blue, hyperindex,backref]{hyperref}
\usepackage[all]{xy}

\theoremstyle{plain}

\newtheorem{theorem}{Theorem}[section]
\newtheorem{lemma}[theorem]{Lemma}
\newtheorem{proposition}[theorem]{Proposition}
\newtheorem{corollary}[theorem]{Corollary}

\newtheorem{definition}[theorem]{Definition}
\newtheorem{example}[theorem]{Example}

\newtheorem{remark}[theorem]{Remark}
\newtheorem{question}[theorem]{Question}

\newcommand{\N}{\mathbb{N}}  
\newcommand{\Z}{\mathbb{Z}}  
\newcommand{\R}{\mathbb{R}}  
\newcommand{\C}{\mathbb{C}}  
\newcommand{\K}{\mathbb{K}}  

\newcommand{\cC}{\mathcal{C}}

\newcommand{\cE}{\mathcal{E}}

\newcommand{\cS}{\mathcal{S}}

\definecolor{bgreen}{rgb}{0.13, 0.55, 0.13}

\renewcommand{\le} {\leqslant}
\renewcommand{\ge} {\geqslant}

\begin{document}

\title{Shift operators and their classification}

\author{Maria Carvalho  \thanks{Partially supported by CMUP, member of LASI, which is financed by national funds through FCT -- Funda\c c\~ao para a Ci\^encia e a Tecnologia, I.P., under the projects with references UIDB/00144/2020 and UIDP/00144/2020, and also supported by the project PTDC/MAT-PUR/4048/2021.} \\ 
\and 
Udayan B. Darji \thanks{Thanks partial support of CMUP and Jagiellonian University  as Jagiellonian Scholar.}\\
\and  
Paulo Varandas 
\thanks{Partially supported by the grant CEECIND/03721/2017 of the Stimulus of Scientific Employment, Individual Support 2017 Call, awarded by FCT.}
}


\date{\today}


\maketitle 


\begin{abstract} 
We introduce a class of linear bounded invertible operators on Banach spaces, called shift operators, which comprises weighted backward shifts and models finite products of weighted backward shifts and dissipative  composition operators. We classify vast families of these shift operators, 
and this classification yields verifiable conditions which we use to construct concrete examples of shift operators with a variety of dynamical properties. As a consequence, we show that, for large classes of shift operators, generalized hyperbolicity is equivalent to  the shadowing property.
\footnote{2020 {\em Mathematics Subject Classification:} Primary 37C50, 47B37, 37B65; Secondary 47B33, 47B01, 

 47A15, 
 37C60, 
 37D99,
 47B80
 
{\em Keywords:} Shift operators, weighted shifts, dissipative operators, generalized hyperbolicity, shadowing}
\end{abstract}

\scriptsize
\normalsize

\section{Introduction}
Linear dynamics is the study dynamical properties of linear operators on infinite dimensional Banach spaces. One of its aims is to classify dynamical properties of linear operators. It is often a non-trivial problem to decide, given a fixed Banach space $X$, if it indeed admits linear operators with dynamical properties such as transitivity, mixing, chaoticity, frequent hypercyclicity, non-trivial invariant measures, etc. For example, it was shown by Ansari \cite{AaronsonMSM1997} and Bernal \cite{Bernal1999PAMS} that every infinite dimensional separable Banach space admits a mixing operator. This was generalized to Fre\'chet spaces by Bonet and Peris, using a key idea of Salas \cite{SalasTAMS1995} which states that transitivity persists under perturbation of the identity by weighted backward shifts.

\smallskip

It is also an important problem to determine whether well-known results of classical dynamics in the setting of compact metric spaces hold in the setting of Banach spaces. For example, one can construct concrete examples which distinguishes the dynamical properties of transitivity, weak mixing, mixing and chaoticity by considering weighted backward shifts. In fact, one characterizes weighted backward shifts with a given dynamical property and then it is simply a matter of choosing proper weights to distinguish these properties. 

\smallskip

The class of weighted backward shifts is also dispensable when studying hyperbolic linear dynamics. For example, what is the relationship between hyperbolicity and the shadowing property? It is well-known that, in finite dimensional Banach spaces, hyperbolicity is equivalent to the shadowing property \cite{OmbachUIAM1994}. However, it was an open problem whether the shadowing property implies hyperbolicity in infinite dimensional spaces. This was settled negatively \cite{BernardesCiriloDarjiMessaoudiPujalsJMAA2018} by constructing a weighted backward shift with such properties. In the process, a new notion of hyperbolicity arose. This notion was extracted from \cite{CiriloGollobitPujals2020} and named generalized hyperbolicity. It turns out that generalized hyperbolicity implies shadowing. However, the converse remains open. A characterization of weighted backward shifts with the shadowing property was given in \cite{BernardesMessaoudiETDS2020}, essentially yielding that for weighted backward shifts, the shadowing property is equivalent to generalized hyperbolicity. It is also known that hyperbolicity is equivalent to expansivity and the shadowing property \cite{BernardesMessaoudiETDS2020,CiriloGollobitPujals2020}.

\smallskip

One can infer from the above that weighted backward shifts are essential in the study of linear dynamics, in the same spirit as symbolic dynamics is essential to the study of topological dynamics and ergodic theory. Despite their immense utility, weighted backward shifts have their shortcomings. For example, weighted backward shifts on $\ell_p$ spaces fail to distinguish between chaoticity and frequent hypercyclicity \cite{BayarRuzsa2015}. To remedy this, they have been generalized in various directions. One of them is a recent investigation of composition operators on $L^P(X, \mu)$, also known as Koopman operators, induced by a bimeasureable map $f\colon L^P(X,\mu) \rightarrow L^P(X,\mu)$. More precisely, when the Radon-Nikodym derivative $\frac{d \,(\mu \circ f)}{d\,\mu}$ is bounded below, the map $ T_f\colon L^p(X, \mu) \rightarrow L^p(X, \mu)$ given by $T(\varphi) = \varphi \circ f$ is a bounded invertible linear operator. This class of composition operators captures weighted backward shifts and numerous results which characterize dynamical properties of $T_f$ have been established. For example, it has been characterized when $T_f$ is transitive or mixing \cite{BayartDarjiPiresJMAA2018} or $T_f$ Li-Yorke chaotic \cite{BernardesDarjiPiresMM2020}. Under the assumption of $f$ being dissipative and satisfying a condition called bounded distortion, it is also possible to characterize when $T_f$ is chaotic, frequently hypercyclic \cite{DarjiPiresProcEdinburgh} and having the shadowing property \cite{DAnielloDarjiMaiurielloJDE}. In fact, under those two assumptions, 
one can show that $T_f$ admits a particular weighted backward shift $B_\omega$ as a factor. Moreover, $T_f$ has one of numerous dynamical properties if and only if $B_\omega$ does \cite{DarioEmmaMartinaJMAA}.

\smallskip

Another important and well-studied generalization of weighted backward shifts on $\ell_p$ is when the underlying space is more general, for example a K\"{o}the space or a Fr\'{e}chet space. Very recently, an intriguing study has been taken up where weighted shifts on directed trees are investigated \cite{GrosseErdmanPAPA2023}.

\smallskip

In this article, we generalize weighted backward shifts in a different direction. We start with a Banach space $X$ and a sequence of suitable linear invertible bounded operators $(S_n)_{n\,\in\, \Z}$ on $X$. We study the weighted shift induced by such operators. Precise definitions are given in Section~\ref{se:main}. In particular, we show that several of the well-studied operators, such as dissipative Koopman operators, fall under the umbrella of our setting. It was shown \cite{DarioEmmaMartinaJMAA} that a shift-like operator admits a weighted shift as a factor; moreover, it was proved that such an operator satisfies a given property $\mathcal P$ if and only if its weighted backward shift factor does, where $\mathcal P$ is one of numerous properties, such as shadowing, expansivity, transitivity, mixing, etc. We will show that our framework also includes this class of operators, meaning that all shift-like operators are conjugate to shift operators.

\smallskip

We classify vast families of these shift operators, including the ones generated by orthogonal matrices, diagonalizable matrices, rotation matrices and hyperbolic matrices. This classification yields verifiable conditions which we use to construct concrete examples of shift
operators with a variety of dynamical properties. 

\smallskip

As stated earlier, it is still an open question whether the shadowing property implies generalized hyperbolicity for bounded invertible linear operators. We show such is the case for a large class of shift operators. 

\smallskip

We would like to point out some  recent developments related to our work. Menet and Papathanasiou investigate weighted shifts on $\ell_p$-sums and $c_0$-sums \cite{menetpapa}. Bernardes and Peris \cite{BernardesPerisAIM2024} carried out a detailed study of various types of shadowing for operators defined on Banach spaces and Fr\'echet spaces. In particular, they showed that finite shadowing and shadowing are not the same for operators on Fr\'echet spaces, although they are equivalent for operators on Banach spaces. Actually, the latter turns out to be more difficult to prove than in the setting of compact spaces. It is still not obvious what the definitions of expansivity and hyperbolicity should be for operators on Fr\'echet spaces. An investigation along this direction is carried out by Bernardes, et al \cite{bernardes2024generalizedhyperbolicitystabilityexpansivity}.

\smallskip

The paper is organized as follows. In Section~\ref{se:main}, we introduce the basic terminology and state the main mathematical contributions and their consequences. In Section~\ref{se:appl} we give applications of our results, including some concrete examples which show how to construct operators with or without certain dynamical properties. Section~\ref{se:proofs} consists of the proofs of the main results. Section~\ref{se:open} states some open problems of interest. 


 \section{Main Results}\label{se:main}

In what follows, given $1 \leq p < +\infty$ and a vector space $Y$ over the field $\K = \R$ or $\C$, we denote by $\ell_p(Y)$ the space of bilateral sequences of elements of $Y$ endowed with the $p-$norm.
\smallskip

Let $X$ be a Banach space and $\mathcal{S} = (S_n)_{n\,\in\,\Z}$ be a sequence of linear bounded invertible operators on $X$ for which there exists $C > 0$ such that 
\begin{equation*}\label{eq:maxnorm}
\max_{n \, \in \, \Z} \, \big\{\|S_n\|, \, \|S_n^{-1}\|\big\} \, < \, C.
\end{equation*}
Given a Banach space ${\bf B} \subset X^\mathbb Z$, the \emph{shift operator} on $\bf B$ generated by $\mathcal{S}$ is the linear bounded invertible operator $\sigma_{\mathcal S}\colon {\bf B} \to {\bf B}$ defined by 
\begin{equation}\label{def:operatorshift}
\sigma_{\mathcal S}\big((x_n)_{n\,\,\in\,\Z}\big) \,=\, \,\big(S_{n+1}(x_{n+1})\big)_{n\,\in\,\Z}.
\end{equation}
This class of operators includes the weighted backward shift. For example, when $X =\R$ and $S_n\colon \R \rightarrow \R $ is given by $S_n(x) = \omega_n x$ for a sequence $\omega = (\omega_n)_{n \, \in \, \Z}$ which is bounded and far away from zero, then we obtain the weighted backward shift that we denote by $B_{\omega}$. 

 \subsection{General properties}

In this section we start illustrating the relevance of the shift operators by describing two classes of linear operators which are modeled by them.

\begin{definition}[Dissipative operators]\label{def:coordinates_projections}
Let $(E_n)_{n\,\in\, \mathbb Z}$ be a collection of closed subspaces of a Banach space $X$. We say that $(E_n)_{n\,\in \,\mathbb Z}$ is a \emph{dissipative decomposition of $X$} if, given $x\in X$, for every $n \in \mathbb Z$ there is $x_n\in E_n$ such that
\begin{itemize}
 \item[(a)] the sums $x^+ = \sum_{n=0}^{+\infty} x_n$ and $x^- = \sum_{n=-\infty}^{-1} x_n$ converge;
 \item[(b)] $x = x^-+x^+$ and this representation is unique;
 \item[(c)] for every $n\in \mathbb Z$, the projection $p_n\colon X \to E_n$, given by $p_n(x) = x_n$, is continuous.
 \end{itemize}

When the previous conditions are valid we write $X = \overline{\text{span }(\bigoplus_{n\,\in\, \mathbb Z} E_n)}$. A linear operator $T\colon X\to X$ is said to be \emph{dissipative} if there exists a closed subspace $E_0\subset X$ such that the collection
$(T^n(E_0))_{n\,\in\, \mathbb Z}$ is a dissipative decomposition of $X$. 
\end{definition}

Dissipative decompositions of Banach spaces appear naturally in the context of linear dynamics. 
This is the case of weighted backward shifts on $\ell_p-$spaces, for $1 \leq p < +\infty$ (cf. \cite{BernardesCiriloDarjiMessaoudiPujalsJMAA2018}), shift-like operators (cf.
\cite{DarioEmmaMartinaJMAA}) and generalized hyperbolic operators on Hilbert spaces with a dense set of periodic orbits (see \cite[Theorem~5]{CiriloGollobitPujals2020}).

\begin{theorem}\label{thm:dissipative2shifts}
Suppose that $T\colon X\to X$ is a dissipative bounded invertible linear operator on a Banach space $X$. Then $T$ is conjugate to a shift operator $\sigma_S\colon \bf B \to \bf{B}$, where $\bf B$ is a Banach space, $\mathcal S = (S_n)_{n\,\in\, \mathbb Z}$ and $S_n$ is the identity map for every $n \in \mathbb Z$. 
\end{theorem}

The proof of Theorem~\ref{thm:dissipative2shifts} does not provide, in general, an explicit description of the Banach space where the shift operator is defined. However, under additional assumptions or within specific frameworks, we may succeed in describing the underlying space. One such framework is the class of dissipative composition operators that were introduced in \cite{DAnielloDarjiMaiurielloJDE,DarjiPiresProcEdinburgh}. It can be easily verified that the definition of dissipative composition operator given below is consistent with our definition of dissipative operators defined above. 

\begin{definition}[Dissipative composition operators]\label{def:shift-like}
Let $f\colon M\to M$ be a bijective bimeasurable map on a $\sigma$-finite measure space $(M, \mathfrak{B}, \mu)$ such that 
the Radon-Nikodym derivative $\frac{d\mu f}{d\mu}$ exists and is bounded from below and from above. This assumption guarantees that the \emph{Koopman operator} $T_f$ induced by $f$, acting on $L^p(\mu)$, for some 
$1 \le p < +\infty$, by
$$\varphi \in L^p(\mu) \quad \quad \mapsto \quad \quad T_f(\varphi) = \varphi\circ f$$
is well-defined, invertible and  $T_f^{-1} = T_{f^{-1}}$.
In case $M = \dot {\bigcup} _{k=-\infty}^{+ \infty} \,f^k (W)$ for some  $W \in \mathfrak{B}$ with $0 < \mu (W) < +\infty$, the operator $T_f$ is called {\em a dissipative composition operator}.  
\end{definition}

We note that the next theorem is not a corollary of the previous one as we obtain $\bf B$ as a specific $L^p$ space.

\begin{theorem}\label{thm:shift-like2shift}
A dissipative composition operator $T_f\colon L^p(\mu) \to L^p(\mu)$, associated with a dissipative system $(M,{\mathfrak{B}},\mu, f)$ generated by a set $W$, is conjugate to a shift operator $\sigma_{\mathcal S}\colon  {\bf B} \to {\bf B}$, where 
$$
{\bf B} \,=\, \Big\{ (\psi_n)_{n\,\in \,\mathbb Z} \in {L^p(\mu\mid_W)}^{\mathbb Z} \colon \sum_{n\,\in\, \mathbb Z}\, \int_W |\psi_n|^p \,\frac{d\mu f^{n}}{d\mu} \, d\mu < +\infty \Big\}
$$ 
endowed with the norm 
$$
\|(\psi_n)_{n\,\in \,\mathbb Z}\|_{\bf B} \,=\, \left(\sum_{n\,\in\, \mathbb Z} \, \int_W |\psi_n|^p \,  \frac{d\mu f^{n}}{d\mu}\, d\mu\right)^{\frac{1}{p}}
$$
and $\mathcal S=(S_n)_{n\,\in\, \mathbb Z}$, with $S_n$ being the identity map on $L^p(\mu\mid_W)$ for every $n \in \Z$. 
\end{theorem}

It is worth mentioning that, even though all the operators $S_n$ are the identity map, the shift operator $\sigma_{\mathcal S}$ can exhibit a variety of dynamical behaviors. For instance, depending on the dissipative system $(M,{\mathfrak{B}},\mu, f)$, the shift operator $\sigma_{\mathcal S}$ may or may not have the shadowing property, be transitive, be expansive, etc (cf. \cite{DarioEmmaMartinaJMAA}).

\subsection{Factors of shift operators}

In what follows we will find weighted backward shifts defined on $\ell_p(\K)$, for $\ \leq p < +\infty$, as factors of shift operators. To this end, we introduce some notation which expedite the  statements of the main results. 

Consider a Banach space $X$ and a sequence $(S_n)_{n \, \in \, \Z}$ of bounded invertible linear operators on $X$. For $n, m \in \Z$, let
\begin{equation}\label{eq:Snm}
 S_{[n, \,m]} \,=\,
\begin{cases}
\begin{array}{cc}
    S_n  \,\circ \,\cdots\, \circ \, S_m & \quad \text{ if $n  < m$}\\
    S_n  & \quad \text{ if $n  = m$}\\
    Id & \quad \text{ if $n  > m$} 
    \end{array}.
\end{cases}
\end{equation}
Given $x\in X\setminus\{0\}$, define the sequence of weights $\omega(x) = (\omega_n(x))_{n\,\in\, \mathbb Z}\in \mathbb R^{\mathbb Z}$ by 

\begin{equation}\label{eq:weights}
\omega_n(x) \,=\,
\begin{cases}
\begin{array}{cc}
\frac{\|(\,S_{[1,n-1]}\,)^{-1}(x)\|}{\|(\,S_{[1,n]}\,)^{-1}(x)\|}  & \quad \text{if} \,\, n > 0 \medskip \\
\frac{\|S_{[n,0]}(x)\|}{\|S_{[n+1,0]}(x)\|}  & \quad  \text{if} \,\, n \leq 0.
\end{array}
\end{cases}
\end{equation}
Since are assuming that there exists $C > 0$ such that 
\begin{equation*}
\max_{n \, \in \, \Z} \, \big\{\|S_n\|, \, \|S_n^{-1}\|\big\} \, < \,C
\end{equation*}
we are sure that the shift operator $\sigma_{\mathcal S}$ is a well-defined bounded invertible operator acting on $\ell_p(X)$, $1\le p <+\infty$, and that the sequence $\big(\omega_n(x)\big)_{n\,\in\, \Z}$ defined by \eqref{eq:weights} is bounded from below by $C$.

\smallskip

We shall now address a special class of sequences $\mathcal S=(S_n)_{n\,\in\, \mathbb Z}$.

\begin{definition}[Orthogonal basis for sequences of operators]\label{def:rig}
Assume that $X$ is a Hilbert space. A sequence $\mathcal S=(S_n)_{n\,\in\, \mathbb Z}$ of bounded invertible linear operators on $X$ \emph{has an orthogonal basis $\cE$} if $\{e_n(x)\colon x \in \cE\}$ is orthogonal for every $n \in \Z$, where 
\begin{equation}\label{eq:vectors}
e_n(x) \,=\,
\begin{cases}
\begin{array}{cc}
\frac{(\,S_{[1,n]}\,)^{-1}\,(x)}{\|(\,S_{[1,n]}\,)^{-1}\,(x)\|}  & \quad \text{if} \,\, n \geq 0 \medskip \\
\frac{S_{[n+1,0]}\,(x)}{\|S_{[n+1,0]}\,(x)\|}  & \quad  \text{if} \,\, n <0
\end{array}.
\end{cases}
\end{equation}
\end{definition}

\medskip

We note that, by construction, if $S_n=L$ for all $n \in \Z$, then $(e_n(x))_{n\, \in\, \Z}$ is the normalized orbit of $x \in \cE$ by $L$. Moreover, in general, one always has 
\begin{equation}\label{eq:EW}
S_{n+1} \,\big(e_{n+1}(x)\big)\,=\, \omega_{n+1}(x)\,e_n(x)  \quad \quad \forall\, n \in Z \quad \quad \forall \, x \in \cE.
\end{equation}

\smallskip

We observe that, if $S_n$ is orthogonal for every $n \in \Z$, then any orthogonal basis $\cE$ of $X$ is an orthogonal basis for the sequence $\mathcal S=(S_n)_{n\,\in\, \mathbb Z}$. In this case, it is known that each $S_n$ is a scalar multiple of an isometry, say $S_n = \alpha_n I_n$, with $\alpha_n\in \R\setminus \{0\}$ and $I_n$ being an isometry. Thus, $\omega_n(x)=|\alpha_n|$ for every $n\in \Z$ and $x \in X\setminus 
\{0\}$. 

\begin{theorem}[Factors of shifts]\label{thm:factors}
Let $X$ be a Hilbert space and suppose that the sequence of operators $\mathcal S=(S_n)_{n\,\in\, \mathbb Z}$ on $X$ has an orthogonal basis $\cE$. Then, for any $x \in \cE$, the weighted backward shift $B_{\omega(x)}\colon \ell_p( \mathbb K) \to \ell_p(\mathbb K)$ is a factor of the shift operator $\sigma_{\mathcal S}\colon \ell_p( X) \to \ell_p( X)$.
\end{theorem}

The following consequence of Theorem~\ref{thm:factors} is due to the fact that, in general, factors preserve transitivity, mixing 
and the shadowing property. 

\begin{corollary}\label{cor:factordynamics}
Suppose we are in the setting of Theorem~\ref{thm:factors}. If $\sigma_{\mathcal S}$ is transitive, mixing or has the shadowing property, then the same holds for $B_{\omega(x)}$ for all $x \in \cE $.
\end{corollary}

We note that, according to  Corollary~\ref{cor:factordynamics}, while proving that $\sigma_{\mathcal S}$ does not have a property in question (transitive, mixing or shadowing), we may simply find $x \in \cE$ such that $B_{\omega(x)}$ does not have the corresponding property (see Example~\ref{ex:no-cones}).

\subsection{Classification theorems}
We next describe the shift operators $\sigma_{\mathcal S}$ when $X$ is a finite dimensional Hilbert space and ${\mathcal S} = (S_n)_{n \, \in \, \Z}$ has an orthogonal basis. Given a product of a finite number of copies of the Banach space $\ell_p(\mathbb K)$, endow the product vector space with the sum norm.

\begin{theorem}[Classification theorem I]\label{thm:productfactor}
Let $X$ be a finite dimensional Hilbert space and assume that $\mathcal S=(S_n)_{n\,\in \,\Z}$ has an orthogonal basis $\cE$. Then, 
 for each $1\le p <+\infty$,  
 $\sigma_{\mathcal S}\colon \ell_p( X) \to \ell_p( X)$ is conjugate to the finite product of weighted backward shifts 
 $$\prod _{x \,\in\, \cE} \,B_{\omega(x)} \colon \quad \big(\ell_p(\mathbb K)\big)^{\dim X} \quad \to \quad \big(\ell_p(\mathbb K)\big)^{\dim X}$$
 where $\dim X$ stands for the dimension of $X$.
\end{theorem}

The reader may wonder why the weights $\omega_n$ are real numbers despite the fact that the underlying field $\K$ may be $\C$. Actually, it is well-known that every weighted backward shift on $\C$ with complex weights is conjugate to a weighted backward shift with positive weights \cite[Section 3.2]{Gimenez-Peris2002}. This also follows from our Theorem~\ref{thm:productfactor}. Moreover, every weighted backward shift with positive weights over $\C$ can easily be seen as conjugate to the product of two weighted backward shifts with positive weights over $\R$.
\smallskip

The assumptions on the previous theorem may seem restrictive. However, below we state an alternative classification theorem for the shift operators $\sigma_{\mathcal S}$ when $X$ is a finite dimensional Banach space and orthogonality is replaced by uniform boundedness of projections. Let us introduce the necessary notions.

\smallskip

Given a Banach space $X$ and a vector $x\in X$, let $F_x\subset X$ denote the subspace spanned by $x$. Consider a basis $\cE$ of $X$, $x \in \cE$ and the projection $\Pi_{x}\colon X\to F_x$ given by
\begin{equation}\label{def:projBanach}
\Pi_{x} \left(\sum_{b\,\in\, \cE} \,\alpha_b \, b 
\right) = \alpha_x \,x   
\end{equation}
for any sequence of scalars $(\alpha_b)_{b\,\in\, \cE}$ in $\mathbb K$. For every $x\in \cE$, the map $\Pi_{x}$ is well defined, since every vector in $X$ has a unique representation in terms of vectors in $\cE$; furthermore, $\Pi_{x}$ is linear, bounded and $\Pi_{x} \circ \Pi_{x} = \Pi_{x}$.

\begin{definition}[$\mathcal S-$bounded projections]\label{def:bounded}
Given a sequence $\mathcal S = (S_n)_{n\,\in\, \Z}$ of bounded invertible linear operators on a Banach space $X$, we say that a basis $\cE$ of $X$ has \emph{$\mathcal S-$bounded projections} if there exists $C>0$ such that 
\begin{equation}\label{eq:boundednorms}
\sup_{n\,\in\, \Z} \,\sup_{b\,\in\, \cE_n}\,\, \|\Pi_{b}\| \,\le \,C \,< \,+\infty,    
\end{equation}
where $\cE_n = \{e_n(x)\colon x\in \cE\}$.
\end{definition}

\begin{theorem}[Classification theorem II]\label{thm:productfactorBanach}
Let $X$ be a finite dimensional Banach space, $\mathcal S=(S_n)_{n\,\in\, \mathbb Z}$ be a sequence of operators on $X$ and $\cE$ be a basis of $X$ with $\mathcal S-$bounded projections. Then, for each $1 \le p < +\infty$, the shift operator $\sigma_{\mathcal S}\colon \ell_p(X) \to \ell_p(X)$ is conjugate to the finite product $\prod _{x \,\in\, \cE}\, B_{\omega(x)} \colon (\ell_p(\K))^{\dim X} \to (\ell_p(\K))^{\dim X}$.
\end{theorem}

Theorem~\ref{thm:productfactorBanach} allows us to strengthen Theorem~\ref{thm:productfactor} by replacing orthogonality with suitable uniform lower and upper bounds on the angles between distinct vectors in each basis $\cE_n$.

\begin{corollary}\label{cor:separatedbasis-2}
Let $X$ be a finite dimensional Hilbert space with dimension $\dim X \geq 2$ and assume that $\mathcal S=(S_n)_{n\,\in \,\Z}$ has a basis $\cE$ of $X$ for which there exists $0 < \gamma < 1/(d-1)$ such that, for every $n \in \Z$, the angle $\measuredangle(u,v)$ between distinct vectors $u,v \in \cE_n$ satisfies the condition
$$ \cos \measuredangle(u,v) \,\in\, [-\gamma, \gamma].$$ 
Then, given $1 \le p < +\infty$, the shift $\sigma_{\mathcal S}\colon \ell_p(X) \to \ell_p(X)$ is conjugate to the finite product of weighted backward shifts $\prod _{x \,\in\, \cE} B_{\omega(x)} \colon \big(\ell_p(\mathbb K)\big)^{\dim X} \to \big(\ell_p(\mathbb K)\big)^{\dim X}$. 
\end{corollary}

We observe that, as $d$ goes to infinity, the requirement in the previous corollary gets closer to orthogonality.

\smallskip

The next corollary provides another way of improving Theorem~\ref{thm:productfactor}, by using Theorem~\ref{thm:productfactorBanach} and a new sufficient condition (which does not depend on the dimension $d$ of the Hilbert space $X$) for a sequence $\mathcal S = (S_n)_{n\,\in\, \Z}$ to have $\mathcal S-$bounded projections. Given a Hilbert space $X$, $v \in X\setminus
\{0\}$ and a subspace $F\subset X$, consider the following infimum 
$$\measuredangle (v,F) \, = \,\inf_{u\,\in\, F\setminus
\{0\}} \; \measuredangle (v,u).$$
For a basis $\cE$ of $X$ and $v \in \cE$, we denote by $F_{n,v}$ the subspace of $X$ generated by the vectors $\big\{e_n(b) \colon b \in \cE \setminus\{v\}\big\}$.

\begin{corollary}\label{cor:angles}
Let $X$ be a finite dimensional Hilbert space with dimension $d \geq 2$ and assume that $\mathcal S = (S_n)_{n\,\in\, \Z}$ has a basis $\cE$ of $X$ satisfying
\begin{equation*}
\inf_{n\,\in\, \mathbb Z, \,\, v \, \in \, \cE} \,\, \measuredangle (e_n(v),F_{n,v}) \,\,> \,\,0.
\end{equation*}
Then, for every $1 \le p < +\infty$, the shift $\sigma_{\mathcal S}\colon \ell_p( X) \to \ell_p( X)$ is conjugate to the finite product of weighted backward shifts $\prod _{x \,\in\, \cE} B_{\omega(x)} \colon \big(\ell_p(\mathbb K)\big)^{\dim X} \to \big(\ell_p(\mathbb K)\big)^{\dim X}$. 
\end{corollary}

\subsection{Equi-properties}

It is known that some relevant properties of linear invertible bounded operators may be conveyed to factors. We will strengthen such results in the setting of Theorem~\ref{thm:productfactor}.

\begin{definition}[Equi-properties]\label{def:equiP}
    Let $X$ be a Banach space and $\mathcal C$ be a class of operators acting on $X$. We say that $\mathcal C$ is:
\begin{enumerate}
    \item[(a)] \emph{Equi-transitive} if for any non-empty open sets $U,V\subset X$ there exists $n \in \N$ so that $T^n(U)\cap V \neq\emptyset$ for every $T\in \mathcal C$.
    \item[(b)] \emph{Equi-mixing} if for any non-empty open sets $U,V\subset X$ there exists $N \in \N$ so that $T^n(U)\cap V \neq\emptyset$ for every $T\in \mathcal C$ and every $n\ge N$.
    \item[(c)] \emph{Equi-shadowing} if there exists $K >0$ such that, for all $T\in \mathcal C$ and any bounded sequence $\{z_n\}_{n \,\in\, \Z}$ in $X$, there is a bounded sequence $\{x_n\}_{n \,\in \,\Z}$ in $X$ satisfying 
    \[ x_{n+1}-T(x_n) = z_n \quad \quad \text{and} \quad \quad  \sup_{n\, \in\, \Z} \,\|x_n\| < K \sup_{n \,\in\, \Z} \,\|z_n\|.\]
  \end{enumerate}
\end{definition}

Observe that, in case $\mathcal C$ consists of a single operator $T$, equi-property $\mathcal P$ is the same as property $\mathcal P$ for $T$. Also notice that, if $\mathcal C$ is finite, then equi-shadowing is equivalent to shadowing for all elements of $\mathcal C$.

\smallskip

In what follows $\mathcal P$ will stand for any of the properties summoned in Definition~\ref{def:equiP}, that is, transitivity, mixing or shadowing.

\begin{corollary}\label{cor:equiP}
Let $X$ be a finite dimensional Banach space, $\mathcal S=(S_n)_{n\,\in\, \mathbb Z}$ be a sequence of operators on $X$ and $\cE$ be a basis of $X$ with $\mathcal S-$bounded projections. If 
$\sigma_{\mathcal S}\colon \ell_p( X) \to \ell_p( X)$ has property $\mathcal P$, then $\mathcal C=\{B_{\omega(x)}\colon x\in \cE\}$ has equi-property $\mathcal P$.
\end{corollary}

Applying Corollary~\ref{cor:equiP}, together with   Theorem~\ref{thm:productfactorBanach} and the fact that mixing, shadowing and generalized hyperbolicity are preserved under finite products and by linear conjugation, and letting ${\mathcal Q}$ be one of these properties, we conclude that:

\begin{corollary}\label{cor:charExpanShadowing}
Let $X$ be a finite dimensional Banach space, $\mathcal S=(S_n)_{n\,\in\, \mathbb Z}$ be a sequence of operators on $X$ and $\cE$ be a basis of $X$ with $\mathcal S-$bounded projections. Then the shift $\sigma_{\mathcal S}\colon \ell_p( X) \to \ell_p( X)$ has property ${\mathcal Q}$ if and only if $B_{\omega(x)}$ has property ${\mathcal Q}$ for all $x \in \cE$.
\end{corollary}

It is easy to verify that Corollary~\ref{cor:charExpanShadowing} fails for the shadowing property in case $X$ is infinite dimensional. For instance, let $X=\ell_2(\mathbb R)$, $\cE = \{e_j \colon j \in \Z\}$ be the canonical basis of $X$ and $T\colon \ell_2(\mathbb R) \to \ell_2(\mathbb R)$ be defined by 
 $$T(e_j) \,=\, \Big(1+\frac1{2^j}\Big) \,e_j$$
 for all $j \in \Z$. Then the shift operator $\sigma_{\mathcal S}\colon \ell_p(X) \to \ell_p(X)$,  determined by $\mathcal S=(S_n)_{n\,\in\, \mathbb Z}$ with $S_n = T$ for every $n \in \Z$, does not have the shadowing property (cf. \cite{BernardesMessaoudiETDS2020}), although for all $x \in \cE$ the weighted backward shift $B_{\omega(x)}\colon \ell_p(\R) \to \ell_p(\R)$ has it.

\subsection{Shadowing \emph{vs.} generalized hyperbolicity}


It is known that generalized hyperbolic linear operators satisfy the shadowing property 
(cf. \cite{CiriloGollobitPujals2020}). On the other hand, a  weighted backward shift on $\ell_p(\mathbb K)$ has the shadowing property if and only if it is generalized hyperbolic (cf. \cite{BernardesMessaoudiETDS2020}); and a finite product of generalized hyperbolic operators is clearly generalized hyperbolic. Since, under the hypothesis of Theorem~\ref{thm:productfactor}, 
$\sigma_{\mathcal S}$ is conjugate to a finite product of weighted backward shifts on $\ell_p(\mathbb K)$, we obtain the following equivalence:

\begin{corollary}\label{cor:GenHypShadowing} 
Let $X$ be a finite dimensional Banach space, $\mathcal S=(S_n)_{n\,\in\, \mathbb Z}$ be a sequence of operators on $X$ and $\cE$ be a basis of $X$ with $\mathcal S-$bounded projections. Then the shift $\sigma_{\mathcal S}\colon \ell_p(X) \to \ell_p(X)$ is generalized hyperbolic if and only if it has the shadowing property.
\end{corollary}

\subsection{Sufficient conditions for shadowing on shift operators}

Applying Bernardes-Messaoudi characterization of the weighted backward shifts with the shadowing property (cf. \cite[Theorem 18]{BernardesMessaoudiETDS2020}) to the setting of Corollary~\ref{cor:charExpanShadowing}, and rewriting $w_n(x)$ in terms of the $S_n$'s, we have the following immediate criterion for shift operators.

\begin{corollary}\label{cor:BM}
Let $X$ be a finite dimensional Hilbert space and assume that the sequence $\mathcal S=(S_n)_{n\,\in\, \mathbb Z}$ has a basis $\cE$ with $\mathcal S-$bounded projections. Then $\sigma_{\mathcal S}\colon \ell_p( X) \to \ell_p( X)$ has the shadowing property if and only if  for every $x \in \mathcal{E}$ one of the following conditions holds: 
\begin{enumerate}
\item[(A)] 
{\scriptsize
$$\lim_{n\,\to\,+\infty} 
                \left(
                \max\left\{ \sup_{k\,\in\, \N} 
        \frac{\|S_{k-1}^{-1} \,\dots \,S_1^{-1} (x)\|}{\|S_{k+n}^{-1} \,\dots \,S_1^{-1} (x)\|},
        \sup_{k \,\ge\, n} \,\left(\frac{\|S_{-k} \,\dots\, S_{0}(x)\|}{\|S_{-k+n}\, \dots\, S_0 (x)\|}\right), 
        \sup_{0\,\le\, k\, < n}\,
        \left(\frac{\|S_{-k} \,\dots\, S_{0}(x)\|}{\|S_{-k+n}^{-1}\, \dots\, S_1^{-1} (x)\|}\right)
        \right\}
        \right)^\frac1n <1.$$}
        \item[(B)] 
        {\scriptsize
                $$\lim_{n\,\to\,+\infty} 
                \left(
                \max\left\{ \inf_{k\,\in\, \N} 
        \frac{\|S_{k-1}^{-1} \,\dots \,S_1^{-1} (x)\|}{\|S_{k+n}^{-1} \,\dots \,S_1^{-1} (x)\|},
        \inf_{k \,\ge\, n} \left(\frac{\|S_{-k} \,\dots\, S_{0}(x)\|}{\|S_{-k+n}\, \dots\, S_0 (x)\|}\right),
        \inf_{0\,\le\, k\, < n}
        \left(\frac{\|S_{-k} \,\dots\, S_{0}(x)\|}{\|S_{-k+n}^{-1}\, \dots\, S_1^{-1} (x)\|}\right)
        \right\}
        \right)^\frac1n >1.$$
        }    
        \item[(C)] 
        {\scriptsize
        $$
        \lim_{n\,\to\,+\infty}\,\,\sup_{k\,\in\, \N}\, \left(\frac{\|S_{-k-n} \,\dots\, S_{0}(x)\|}{\|S_{-k+1}\, \dots\, S_0 (x)\|}\right)^\frac1n < 1
        \quad \quad \quad   \text{and} \quad \quad \quad  
        \lim_{n\,\to\,+\infty}\, \,\inf_{k\,\in\, \N}
        \left(\frac{\|S_{k-1}^{-1} \,\dots \,S_1^{-1} (x)\|}{\|S_{k+n}^{-1} \,\dots \,S_1^{-1} (x)\|}
        \right)^\frac1n > 1. 
        \textcolor{white}{\hspace{1.6cm} aaaaaaaa}
        $$ }
\end{enumerate}
\end{corollary}

\medskip

A shift operator $\sigma_{\mathcal{S}}$ may fail to satisfy the shadowing property even though each of the operators $S_n$ in the sequence $\mathcal S$ has that property, as the next example illustrates. 

\begin{example}[Operators without shadowing]\label{ex:no-cones}
\emph{Consider the Banach space $\mathbb R^2$ endowed with the Euclidean norm and 
$\ell^\infty(\mathbb R^2) \,=\, \big\{(x_n, y_n)_n \in {(\mathbb{R}^2)}^{\mathbb Z} \colon \, \sup_{n\,\in\, \mathbb Z} \,\|(x_n, y_n)\| \,<\,+\infty \big\}.$
Let $T\colon \mathbb R^2 \to \mathbb R^2$ be the linear invertible bounded operator given by 
$$(x,y)\in \mathbb R^2 \quad \quad \mapsto \quad \quad T(x,y)\,=\,\Big(2x, \frac12 y\Big).$$
The map $T$ is hyperbolic, hence satisfies the shadowing property. Consider now the shift operator $\sigma_\mathcal{S}$ where $\mathcal{S}=(S_n)_{n\,\in\, \mathbb Z}$ is defined by $S_n=T$ if $n$ is odd and $S_n=T^{-1}$ otherwise. This shift operator does not have the shadowing property in view of Corollary~\ref{cor:BM}.}
\end{example}


 \section{Applications}\label{se:appl}

To illustrate the scope of Theorems~\ref{thm:productfactor} and \ref{thm:productfactorBanach}, in this section we will discuss several examples of shift operators exhibiting a variety of dynamics.

\subsection{Sequences of operators with an orthogonal basis}

\begin{example}[Rotation operators]\label{ex:rot}
\emph{Consider a rotation matrix in $X = {\mathbb R}^2$ given by
\[ R_\theta =
\begin{pmatrix}
\cos \,(2\pi\theta) & -\sin \,(2\pi\theta)  \\
\sin \,(2\pi\theta) & \,\,\,\,\cos \,(2\pi\theta) 
  \end{pmatrix}
\]
where $\theta \in \,\,]0,1[$. Let $\cS = (S_n)_{n \,\in \,\Z}$ where $S_n = 1/2\,R_{\theta_n}$ for all $n \in \Z$ and $\theta_n \in \,\,]0,1[$. Then any orthogonal basis of $\R^2$ is an orthogonal basis for $\cS$. By Theorem~\ref{thm:productfactor}, we know that $\sigma_{\cS}\colon \ell_p({\mathbb R}^2) \to \ell_p({\mathbb R}^2)$ is conjugate to the product of weighted backward shifts $B_{\omega} \times B_{\omega}$, where $\omega=(w_n)_{n\,\in\, \Z}$ is given by $\omega_n=\frac12$ for every $n\in \Z$.
From Corollary~\ref{cor:charExpanShadowing}
we also conclude that $\sigma_{\mathcal S}$ has the shadowing property, and is in fact hyperbolic. }
\end{example}

\begin{example}[Diagonal operators]\label{ex:diagonal}
\emph{Consider $d \in \N$, $X = {\mathbb K}^d$ and invertible diagonal matrices
        \[S_n=
  \begin{pmatrix}
    \lambda_n(1) & 0 & \dots & 0 \\
    0 & \lambda_n(2)  & \dots & 0 \\
    \vdots & \vdots & \ddots & 0 \\
    0 & 0 & \dots & \lambda_n(d)
  \end{pmatrix}.\] 
Clearly, the canonical basis of $\mathbb \K^d$ is an orthogonal basis for the sequence $\cS = (S_n)_{n\, \in\, \Z}$. By Theorem~\ref{thm:productfactor}, the shift $\sigma_{\cS}\colon \ell_p({\mathbb K}^d) \to \ell_p({\mathbb K}^d)$ is conjugate to a finite product of weighted backward shifts. Moreover, $\sigma_{\cS}$ has the shadowing property if and only if $(\lambda_n(t))_{n\, \in\, \Z}$ satisfies Bernardes-Messaoudi criteria \cite[Theorem 18]{BernardesMessaoudiETDS2020} for every $1 \le t \le d $.
Furthermore, $\sigma_{\mathcal S}$ has the mixing property if and only if 
$(\lambda_n(t))_{n\, \in\, \Z}$ satisfies the criteria in \cite[Example 4.9]{GrosseErdmannMaguillot2011}.}

\smallskip

\emph{We note that, conversely, a product of $d$ weighted backward shifts, each defined in $\ell_p({\mathbb K})$, is conjugate to the shift operator $\sigma_\cS$ generated by the sequence $(S_n)_{n\, \in\, \Z}$ of diagonal matrices, with respect to the canonical basis in ${\mathbb K}^d$, whose entries are precisely the weights of the factors.}
\end{example}

\begin{example}[Operators with an orthogonal basis of eigenvectors]\label{ex:diagonalizable} \emph{Let $X = {\mathbb R}^2$ and consider 
\[ L =
  \begin{pmatrix}
    2 & 1  \\
    1 & 1  
  \end{pmatrix}.
\]
The matrix $L$ is not orthogonal, yet it is diagonalizable with eigenvalues $\frac{3\pm \sqrt{5}}2$ and eigenvectors $\cE = \big\{\big(1,\frac{\sqrt{5}-1}{2}\big), \,\big(1,-\frac{\sqrt{5}+1}{2}\big)\big\}$, which are orthogonal. Thus, the sequence $\mathcal S=(S_n)_{n\,\in\, \mathbb Z}$ of operators in ${\mathbb R}^2$ given by $S_n=L$ for all $n \in \Z$ has an orthogonal basis. Therefore, by Theorem~\ref{thm:productfactor}, the shift $\sigma_{\cS}\colon \ell_p({\mathbb R}^2) \to \ell_p({\mathbb R}^2)$ is conjugate to the product of weighted backward shifts $B_{\omega} \times B_{\bar \omega}$, where $\omega=(\omega_n)_{n\,\in\, \Z}$, $\bar \omega=(\bar \omega_n)_{n\,\in\, \Z}$, $\omega_n=\frac{3+ \sqrt{5}}2$ and $\bar \omega_n=\frac{3- \sqrt{5}}2$ for any $n\in \Z$. Therefore,  $\sigma_{\mathcal S}$ is hyperbolic.}
\end{example}

\subsection{Operators without an orthogonal basis}

We now address a new class of operators to obtain dynamical information in case the sequence of operators does not have an orthogonal basis.

\subsubsection{Jointly diagonalizable operators}
\label{subsec:diagonalizable}

\begin{definition}[Jointly diagonalizable]
Let $X$ be a finite dimensional Banach space. A sequence ${\mathcal S} = (S_n)_{n \,\in\, \Z}$ of operators on $X$ is said to be \emph{jointly diagonalizable} if there is a linear invertible bounded operator $\mathcal L$ such that, for every $n \in \Z$,  there exists a diagonal operator $D_n$ satisfying $S_n = \mathcal L D_n {\mathcal L}^{-1}$.
\end{definition}

\begin{proposition}\label{cor:jointlydiagonable} Let $X$ be a Banach space with dimension $1 \leq d < +\infty$. Suppose that the sequence ${\mathcal S} = (S_n)_{n \,\in\, \Z}$ of operators on $X$ is jointly diagonalizable. Then the shift $\sigma_{\mathcal S}\colon \ell_p(X) \to \ell_p(X)$ is conjugate to the product of $d$ weighted backward shifts, each one defined on $\ell_p(\mathbb K)$. 
\end{proposition}

We start by proving a general lemma regarding the conjugation of shift operators.

\begin{lemma}\label{le:coordinate-conjugacy}
Let $X,Y$ be Banach spaces, $\cS = (S_n)_{n\, \in\, \Z}$  be a sequence of linear invertible bounded operators on $X$ and $\mathcal T = (T_n)_{n\, \in\, \Z}$ be a sequence of linear invertible bounded operators on $Y$. Assume that ${\bf B}_X \subset X^{\mathbb Z}$ and ${\bf B}_Y \subset Y^{\mathbb Z}$ are Banach spaces such that $\sigma_{\mathcal S}\colon {\bf B}_X \to {\bf B}_X$ and $\sigma_{\mathcal T}\colon {\bf B}_Y \to {\bf B}_Y$ are well defined shift operators. In addition, suppose that there exists a bounded invertible linear operator $H\colon X \to Y$ such that $H\circ S_n = T_n \circ H$ for every $n \in \Z$, and $H^{\mathbb Z}({\bf B}_X)={\bf B}_Y$. Then $\sigma_{\mathcal S}$ is conjugate to $\sigma_{\mathcal T}$ by the map $\mathcal{H}\colon {\bf B}_X \to {\bf B}_Y$, where $\mathcal{H}((x_n)_{n\,\in\, \mathbb Z}) = (H(x_n))_{n\,\in\, \Z}$.
\end{lemma}

\begin{proof} We note that $\mathcal{H}$ is precisely the restriction of $H^{\Z}$ to ${\bf B}_X$. Therefore, as $H$ is a linear bounded invertible operator and $H^{\mathbb Z}({\bf B}_X)={\bf B}_Y$, so is $\mathcal{H}\colon  {\bf B}_X \to {\bf B}_Y$. Moreover, for any $(x_n)_{n\,\in\, \Z} \in {\bf B}_X$, on has
\begin{align*}
\mathcal{H}\big(\sigma_{\mathcal S}((x_n)_{n\,\in\, \Z})\big)
& = \,\mathcal{H}\big((S_{n+1}(x_{n+1}))_{n\,\in\, \Z}\big)
\,=\, \Big(H\big(S_{n+1}(x_{n+1})\big)\Big)_{n\,\in\, \Z} \\
& = \,\big(T_{n+1}(H(x_{n+1}))\big)_{n\,\in\, \Z} 
\,=\, \sigma_{\mathcal T}\Big(\big(H(x_n)\big)_{n\,\in\, \Z}\Big)\\
& = \,\sigma_{\mathcal T} \Big(\mathcal{H}\big((x_n)_{n\,\in\, \Z}\big)\Big).
\end{align*}
\end{proof}

\begin{proof}[Proof of Proposition~\ref{cor:jointlydiagonable}]
    
Assume that $X$ is a finite dimensional Banach space and let ${\mathcal S} = (S_n)_{n \,\in\, \Z}$ be a sequence of operators on $X$ jointly diagonalizable by a linear invertible operator $\mathcal L\colon X \to X$. For every $n \in \Z$, take the diagonal operator $D_n$ satisfying $S_n = \mathcal L D_n {\mathcal L}^{-1}$. Consider the sequence ${\mathcal D} = (D_n)_{n \,\in\, \Z}$ of diagonal operators on $X$ and the corresponding shift operator $\sigma_{\mathcal D}$. Then, by Lemma~\ref{le:coordinate-conjugacy}, the operators $\sigma_{\mathcal S}\colon \ell_p(X) \to \ell_p(X)$ and $\sigma_{\mathcal D}\colon \ell_p(X) \to \ell_p(X)$ are conjugate. Moreover, by Theorem~\ref{thm:productfactor} (cf. Example~\ref{ex:diagonal}), the shift $\sigma_{\mathcal D}$ is conjugate to a finite product of weighted backward shifts. Hence, so is $\sigma_{\mathcal S}$. 
\end{proof}

\begin{example}[Jointly diagonalizable operators]\label{ex:jdo}
\emph{Let $X= {\mathbb R}^2$ and consider the sequence $\cS = (S_n)_{n\, \in\, \Z}$ of operators in $X$ such that $S_n = A$ for every $n \in \Z$, where
\[ L \,= \,\begin{pmatrix}
2 & 3\\
1 & 2 \\
\end{pmatrix}. \]
The eigenvalues of $L$ are $\lambda = 2 \pm \sqrt{3}$ with eigenvectors $(\pm \sqrt{3}, 1)$, respectively. As the eigenvectors are not orthogonal and the matrix $L$ is hyperbolic, the image by $L$ of any two orthogonal vectors is not orthogonal. Consequently, ${\mathcal S} = (S_n)_{n \,\in \,\Z}$ does not have an orthogonal basis. However, as $(S_n)_{n \, \in \,\Z}$ is jointly diagonalizable, by Proposition~\ref{cor:jointlydiagonable} we still conclude that $\sigma_{\mathcal S}$ is conjugate to the product of two weighted backward shift, and is hyperbolic.} 
\end{example}

\subsubsection{Anosov family}

In Subsubsection~\ref{subsec:diagonalizable} we assumed a joint diagonalization condition on the sequence $(S_n)_{n \, \in \,\Z}$. In what follows, inspired by the shadowing property for non-autonomous Anosov diffeomorphisms which need not admit a common invariant splitting \cite{CastroRodriguesVarandas}, we provide an application of Theorem~\ref{thm:productfactorBanach} where such requirement is relaxed.

\begin{example}[Anosov operators]\label{ex:randomAnosov}
\emph{Let $X = {\mathbb R}^2$ and consider the hyperbolic, symmetric matrix (hence normal) 
\[ L =
\begin{pmatrix}
2 & 1  \\
1 & 1  
\end{pmatrix}.\]
Let $0 < \lambda_-< 1 <\lambda_+$ be the eigenvalues of $L$ with eigenvectors $v^-$ and $v^+$, respectively. Let $\mathbb R^2 = E^+\oplus E^-$ be the orthogonal $L-$invariant splitting given by the eigenspaces $E^+$ and $ E^-$ of $L$. In particular, there exist convex cones $\mathcal C^+$ and $\mathcal C^-$ containing $E^+$ and $E^-$, respectively, such that:
\begin{itemize}
\item[(i)] $\mathcal C^+\cap \mathcal C^- \,=\,\{0\}$
\item[(ii)] $L(\mathcal C^+) \subsetneq \mathcal C^+$ 
\item[(iii)] $L^{-1}(\mathcal C^-) \subsetneq \mathcal C^-$
\item[(iv)] there exists $\eta > 1$ such that 
$$\|L(v)\| \,>\, \eta \,\|v\| \quad  \forall\, v\in \mathcal C^+ \quad \quad \text{and} \quad \quad  
\|L^{-1}(v)\|\, > \,\eta\, \|v\| \quad \forall\, v\in \mathcal C^-.$$
\end{itemize}}

\emph{In what follows we will choose sequences of matrices whose compositions will preserve the previous cones. Let $\mathcal U$ be a small open neighborhood of $L$ in such a way that, for each $T\in \mathcal U$, one has
\begin{equation}\label{eq:conesB0}
T(\mathcal C^+) \subsetneq \mathcal C^+\, \qquad 
\text{and} \qquad 
T^{-1}(\mathcal C^-) \subsetneq \mathcal C^-
\end{equation} 
and
\begin{equation*}\label{eq:conesB}
\|T(v)\| \,\ge\,  \eta\, \|v\| 
\quad \forall\, v\in \mathcal C^+ \quad \quad \text{and}
\quad \quad \|T^{-1}(v)\| \,\ge\, \eta\, \|v\|
\quad \forall\, v\in \mathcal C^-.
\end{equation*}
Regarding the existence of such an open neighborhood of $L$ we refer the reader to \cite{katok_hasselblatt_1995}.}
\medskip

\emph{Let $\cS = (S_n)_{n \,\in\, \Z}$ be a sequence of bounded invertible linear operators on $\mathbb R^2$ such that $S_0=Id$ and 
\begin{equation*}\label{eq:choiceAnosov}
S_{-n} \in \mathcal U \quad \text{and}\quad 
S_{n} \in \mathcal U \quad \quad  \forall \,n\in \N.
\end{equation*}
This choice, together with the definition of the vectors $e_n(\cdot)$ in \eqref{eq:vectors} and the inclusions  \eqref{eq:conesB0}, imply that, for any vectors $b^+\in \cC^+$ and $b^-\in \cC^-$, one has
\begin{equation*}
e_n(b^+) \,=\,
\frac{S_{[n+1,0]}\,(b^+)}{\|S_{[n+1,0]}\,(b^+)\|} \in \,\cC^+ \qquad   \forall \,\, n <0
\end{equation*}
and 
\begin{equation*}
e_n(b^-) \,=\,
\frac{(\,S_{[1,n]}\,)^{-1}\,(b^-)}{\|(\,S_{[1,n]}\,)^{-1}\,(b^-)\|} \in \,\cC^-  \qquad \forall \,\, n \geqslant 0. 
\end{equation*}}
\medskip

\emph{We proceed by constructing special vectors $v_*^+\in \cC^+$ and $v_*^-\in \cC^-$ which additionally satisfy $$e_n(v_*^+)\in \,\cC^+ \quad  \forall\, n \ge 0 \quad\quad  \text{and} \quad \quad e_n(v_*^-)\in \,\cC^- \quad \forall \,
n < 0.$$ 
Consider the splitting
$\mathbb R^2 = E^+_*\oplus E^-_*$, where the subspaces $E^+_*$ and $E^-_*$ are defined by
\begin{eqnarray*}
E^+_* &=& \bigcap_{n \, \ge\, 1}\, S_{1}\dots S_{n} (\mathcal C^+) \,\,\subset\,\, \mathcal C^+ \\
E^-_* &=& \bigcap_{n\,< \,0} \, 
S_{0}^{-1} \dots S_{n+1}^{-1} (\mathcal C^-) \,\,\subset\,\, \mathcal C^-.
\end{eqnarray*}
As $X=\R^2$, there are unit vectors $v_*^\pm \in E_*^\pm\setminus\{0\}$. If  one considers the basis $\cE=\{v_*^-,v_*^+\}$ of $\mathbb R^2$ then, by construction,  
\begin{equation*}\label{eq:contrscones}
S_{n+1}
\dots S_{-1} S_0 \,(v_*^-) \in \mathcal C^-\quad  
\forall n < 0
    \quad \quad  \text{and}\quad \quad 
S_{n}^{-1} \dots S_{1}^{-1}\,(v_*^+)  \in \mathcal C^+ \quad \forall n \ge 1. 
\end{equation*}
Consequently, $e_n(v^\pm_*) \in \cC^\pm$ for every $n\in \mathbb Z$.}

\smallskip

\emph{We claim that the basis $\cE$ has $\mathcal S-$bounded projections. Indeed, given $\alpha, \beta \in \R$ satisfying $|\alpha|+|\beta| >0$, one has 
\begin{align*}
\|\alpha \,e_n(v_*^-) + \beta\, e_n(v_*^+)\|^2
& \,=\, \big\langle \alpha\, e_n(v_*^-) + \beta \,e_n(v_*^+),\,\, \alpha \,e_n(v_*^-) + \beta \,e_n(v_*^+) \big\rangle    \\
& \,=\, \alpha^2  \, \|e_n(v_*^-)\|^2 
    + 2\alpha\beta \, \big\langle e_n(v_*^-), \,e_n(v_*^+)\big\rangle  + \beta^2 \, \|e_n(v_*^+)\|^2 \\
& \,=\, \alpha^2  + 2\alpha\beta \, \big\langle e_n(v_*^-), \,e_n(v_*^+)\big\rangle  + \beta^2 \\
& \,\ge\,  \alpha^2 + \beta^2 -  2 \,|\alpha|\, |\beta|\,
    |\cos \theta| \\
& \,\ge\,  \big(1-|\cos \theta|\big) \;\big(\alpha^2 + \beta^2\big) 
\end{align*}
where $\theta \in \,\,]0,\pi[$ is a (uniform) lower bound for the angle between the vectors $e_n(v_*^-)$ and $e_n(v_*^+)$ for all $n \in Z$, whose existence is ensured by the separation of the cones $\cC^+$ and $\cC^-$. Consequently, given $n \in \Z$, if $x = \alpha \,e_n(v_*^-) + \beta\, e_n(v_*^+)$ then
\begin{align*}\label{ex:boundednormsBanachAnosov}
\sup_{x\,\in\, \mathbb R^2\,\setminus\, \{0\}} \,\,
 \frac{\|\Pi_{e_n(v_*^+)}(x)\|}{\|x\|} 
 & \,=\,  \sup_{|\alpha|\,+\,|\beta|\,>\,0}\,
 \frac{|\beta|}{\|\alpha \,e_n(v_*^-) + \beta\, e_n(v_*^+)\|} \\
& \,\le\, \frac1{\sqrt{1-|\cos \theta|}} \,\,
    \sup_{|\alpha|\,+\,|\beta|\,>\,0 }\,
    \frac{|\beta|}{\big(\alpha^2 + \beta^2\big)^\frac12} \\
& \,\le\, \frac1{\sqrt{1-|\cos \theta|}}.
\end{align*}    
Similar estimates are valid for the projections $\Pi_{e_n(v_*^-)}$. Therefore,  
\begin{equation*}
\sup_{n\,\in\, \Z} \,\sup_{b\,\in\, \cE} \,\sup_{x\,\in\, \mathbb R^2\,\setminus\, \{0\}}\,
  \frac{\|\Pi_{e_n(b)}(x)\|}{\|x\|} \,< \,+\infty,
    \end{equation*}
proving the claim. }

\smallskip

\emph{Now, using Theorem~\ref{thm:productfactorBanach}, we conclude that the shift operator $\sigma_{\mathcal S}$ is conjugate to the product map $B_{\omega(v_*^-)} \times B_{\omega(v_*^+)}$. Furthermore, the weighted backward shifts $B_{\omega(v_*^-)}$ and $B_{\omega(v_*^+)}$ are hyperbolic. In fact, summoning \eqref{eq:EW}, we obtain
$$
\omega_{n}(v_*^+) \,=\, \frac{\|\,S_{n+1}\big(e_{n+1}(v_*^+)\big)\|}{\|e_{n}(v_*^+)\|}
\,= \, \|\,S_{n+1}\big(e_{n+1}(v_*^+)\big)\| \,\ge\, \eta \, \|e_{n+1}(v_*^+)\| \,=\, \eta \,> \, 1 \quad \quad \forall\, n \in \mathbb Z
$$ 
and, similarly, since 
$$\big\|S_{n+1}^{-1}\,\big(S_{n+1} \,\big(e_{n+1}(v_*^-)\big)\big)\big\| \,\ge\, \eta \,\big\|S_{n+1} \,\big(e_{n+1}(v_*^-)\big)\big\| \quad \quad \forall\, n\in \Z$$
one gets
$$\omega_{n}(v_*^-) \,=\, \frac{\|\,S_{n+1}\big(e_{n+1}(v_*^-)\big)\|}{\|e_{n}(v_*^-)\|}
\,= \, \|\,S_{n+1}\big(e_{n+1}(v_*^-)\big)\|
\,\le\, \eta^{-1} \, < \, 1 \quad \quad \forall\, n \in \mathbb Z. $$
In particular, by Corollary~\ref{cor:BM}, the shift $\sigma_{\mathcal S}$ satisfies the shadowing property.}
\end{example}

\subsubsection{Sequences of elliptic and hyperbolic matrices}
In this subsection we will provide examples of shift operators $\mathcal S = (S_n)_{n\,\in\, \mathbb Z}$ so that, despite the operators $S_n$ fail to preserve common cones, the shift $\sigma_{\mathcal S}$ is still conjugate to a finite product of weighted backward shifts. Furthermore, some of these examples satisfy the shadowing property while others do not.   

\begin{example}[Elliptic and hyperbolic matrices]\label{ex:hyp-ellliptic}
\emph{Consider the linear operators on $\mathbb R^2$ given by the matrices
$$
R_{2\pi\zeta}\, = \left(
\begin{array}{cc}
     \cos \,(2\pi\zeta) & -\sin \,(2\pi\zeta)   \\
      \sin \,(2\pi\zeta) & \,\,\,\cos \,(2\pi\zeta) 
\end{array}
\right), \quad \zeta \in \mathbb R\setminus \mathbb Q
\qquad \text{and} \qquad 
L \,= \left(
\begin{array}{cc}
     2 & 0  \\
      0 & \frac12   
\end{array}
\right).
$$
The matrix $R_{2\pi\zeta}$ determines an irrational rotation of the plane, whereas $L$ is a hyperbolic matrix. Thus, matrices $R_{2\pi\zeta}$'s neither preserve any common proper subspaces of $\mathbb R^2$ nor admit a common proper invariant cone. In what follows, we will consider sequences $\mathcal S=(S_n)_{n\,\in\, \mathbb Z}$ of operators on $\R^2$ where $S_0=Id$ and, for each $n \in \Z \setminus \{0\}$, $S_n$ is either $R_{2\pi\zeta}$ or $L$, chosen in a way that the shift operator $\sigma_{\mathcal S}$ is conjugate to a product of weighted backward shifts.}

\smallskip

\emph{Since $L$ is hyperbolic, there exist convex cones $\mathcal C^+$ and $\mathcal C^-$ satisfying the properties (i)-(iv) as in Example~\ref{ex:randomAnosov}. On the other hand, given  an irrational $\zeta$, there exists a strictly increasing sequence $(q_\ell)_{\ell\,\in\, \N}$ of positive integers such that 
$ \lim_{\ell \, \to \, +\infty}\, q_\ell \,\zeta \, (\text{mod }1) \,=\, 1$. 
Therefore, the sequence of matrices $\big(R_{2\pi\zeta}^{q_\ell}\big)_{\ell \, \in \, \N}$ converges, as $\ell$ goes to $+\infty$, to the identity matrix. Moreover, taking a subsequence if necessary, we may assume without loss of generality that, for every $\ell \in \N$, 
\begin{equation}\label{eq:invariance-1}
R_{2\pi\zeta}^{q_\ell} \,(\, L(\cC^+)\,)\, \subsetneq \,\mathcal \cC^+ \quad \quad \text{and} \quad \quad 
L(R_{2\pi\zeta}^{q_\ell} \,(\, \cC^+\,))\, \subsetneq \,\mathcal \cC^+
\end{equation}
\begin{equation}\label{eq:invariance-2}
L^{-1}(R_{2\pi\zeta}^{-q_\ell} \,(\, \cC^-\,))\, \subsetneq \,\mathcal \cC^- \quad \quad \text{and} \quad
\quad R_{2\pi\zeta}^{-q_\ell} \,(\, L^{-1}(\cC^-\,)))\, \subsetneq\, \mathcal \cC^-.
\end{equation}}

\smallskip

\emph{We will now consider sequences $\mathcal S=(S_n)_{n\,\in\, \mathbb Z}$ of the form
\begin{equation}\label{eq:specialchoiceSn}
\dots \underbrace{L \dots L}_{n_{2}} \;  \underbrace{R_{2\pi\zeta} \dots R_{2\pi\zeta}}_{m_{1}} \; \underbrace{L \dots L}_{n_1} \, \overset{\,\,\downarrow}{Id}  \, \underbrace{L \dots L}_{n_1}  \; \underbrace{R_{2\pi\zeta} \dots R_{2\pi\zeta}}_{m_1} \; \underbrace{L \dots L}_{n_2} \dots 
\end{equation}
for suitable choices of sequences of positive integers $(n_i)_{i\, \in\, \N}$ and $(m_i)_{i\,\in \, \N}$, where $\downarrow$ marks the position zero. Given any such sequence $\mathcal S$ of linear operators,
consider the splitting
$\mathbb R^2 = E^+_*\oplus E^-_*$, where the subspaces $E^+_*$ and $E^-_*$ are defined by
\begin{eqnarray*}
E^+_* &=& \bigcap_{\ell \, \ge \,1}\, L^{n_1} R_{2\pi\zeta}^{m_1} \dots L^{n_\ell} R_{2\pi\zeta}^{m_\ell} \,\,(\mathcal C^+) \,\subset\, \mathcal C^+ \\
E^-_* &=& \bigcap_{\ell\, \ge\, 1} \, 
 L^{-n_1} R_{2\pi\zeta}^{-m_1} \dots L^{-n_\ell} R_{2\pi\zeta}^{-m_\ell} \,\,(\mathcal C^-) \,\subset\, \mathcal C^-.
\end{eqnarray*}
If $v_*^\pm \in E_*^\pm\setminus\{0\}$ are unit vectors and one considers the basis 
$\cE = \{v_*^-,v_*^+\}$ of $\mathbb R^2$ then, by construction, as the cones $\mathcal C^-$ and $\mathcal C^+$ are sufficiently separated and invariant (cf. \eqref{eq:invariance-1} and \eqref{eq:invariance-2}), there exists $\theta>0$ such that 
$$\measuredangle (e_n(v_*^+), e_n(v_*^-))\,\ge \,\theta\,>\,0 \qquad \forall \,n\in \Z.$$
In particular, due to Corollary~\ref{cor:separatedbasis-2}, $\cE$ has $\mathcal S-$bounded projections.
Therefore, by Theorem~\ref{thm:productfactorBanach},
the shift operator $\sigma_{\mathcal S}$ is conjugate to the product $B_{\omega(v_*^+)} \times B_{\omega(v_*^-)}$. 
}
\end{example}

\begin{example}[Selection of the frequencies $(n_i)_{i\, \in\, \N}$ and $(m_i)_{i\,\in \, \N}$]\label{ex:hyp-ell-shadowing}
\emph{We may change the frequency with which $L$ occurs in the sequence $(S_n)_{n\,\in\, \mathbb Z}$, as illustrated in \eqref{eq:specialchoiceSn}, to study its impact on the properties of the shift operator $\sigma_{\mathcal S}$. We consider the following two cases.}

\medskip

\noindent \emph{$(1)$} Unbounded gaps: $\, \sup_{i\, \in \, \N} \,m_i \,= \,+\infty$.
\medskip

\noindent \emph{In this case, for every $k\in \Z$,
$$\limsup_{n\,\to\,+\infty}\,\, {\|S_{-k-n} \,\dots\, 
S_{-k}(v_*^{-})\|}^\frac1n \,> \,1 \qquad   \text{and} \qquad 
\liminf_{n\,\to\,+\infty}\, \, 
{\|S_{k+n}^{-1} \,\dots \,S_k^{-1} (v_*^{+})\|}^{-\frac1n} \,< \,1. $$
 However, since there exist sufficiently large numbers of consecutive occurrences of rotation matrices $R_{2\pi\zeta}$ in the sequence $\mathcal S=(S_n)_{n\,\in\, \Z}$, we conclude that, for any $n\in \N$ there is $k\in \N$ such that, for all $x \in \R^2 \setminus \{0\}$,
$$\|S^{-1}_{k+n} \dots S_k^{-1} (x)\| \,=\, \|S_{-k-n} \dots S_{-k} (x)\| \, = \, \|x\|.$$
This way, the sequence $\mathcal S$ does not satisfy  
condition (C) in Corollary~\ref{cor:BM}.
The presence of arbitrary large numbers of consecutive occurrences of the elliptic matrices $R_{2\pi\zeta}$ also prevent conditions (A) and (B) in 
Corollary~\ref{cor:BM} to hold.
Therefore, $\sigma_{\mathcal S}$ does not satisfy the shadowing property.}
\medskip 

\noindent
\emph{$(2)$ } Bounded gaps: $\, \sup_{i\, \in \, \N} \,m_i \ < \,+\infty$.

\medskip

\noindent \emph{Let $M = 1+\sup_{i\, \in \, \N} \,m_i$. We note that 
\begin{equation}\label{eq:fulldensity1B}
  \inf_{k \, \in\, \N}\,\,\frac{\#\,\big\{k \le j \le k+M \colon \,S_j = L \big\}}{M} \,\geq\, \frac{1}{M}.
\end{equation}
Using \eqref{eq:EW}, we know that, for any $n\in \mathbb Z$,
$$\omega_{n}(v_*^+) \,= \, \|S_{n+1}\,\big(e_{n+1}(v_*^+)\big)\| \quad
\text{and} \quad \omega_{n}(v_*^-) \,= \, \|S_{n+1}\,\big(e_{n+1}(v_*^-)\big)\|$$
and so, whenever $S_{n+1} = L$, one has
$$\omega_{n}(v_*^+) \,\ge \,\eta\, > \, 1 \quad \text{and}\quad \omega_{n}(v_*^-) \,\le\, \eta^{-1}\,< \, 1 \quad \quad \forall\, n \in \Z;$$
otherwise, $S_{n+1} = R_{2\pi\zeta}$ and
$$\omega_{n}(v_*^+) \,=\,\omega_{n}(v_*^-)\,=\,1.$$ 
Therefore, for every $k \in \N$ and $n \in \N$,  
$$\big|\omega_k(v_*^+) \dots \omega_{k+n}(v_*^+)\big|^{\frac1n} \,\,=\,\, \eta^{\,\,\frac{\#\,\big\{k \,\le \,j \,\le \,k+n \colon \,\,S_j = L \big\}}{n}}$$
which yields
\begin{equation}\label{eq:aux}
\inf_{k\,\in\, \N} \,\big|\omega_k(v_*^+) \dots \omega_{k+n}(v_*^+)\big|^{\frac1n} \,=\inf_{k\,\in\, \N} \,\eta^{\,\frac{\#\,\big\{k\, \le\, j \,\le\, k+n \colon \,S_j = L \big\}}{n}}\,\geq\,\eta^{\,\inf_{k\,\in\, \N} \,\frac{\#\,\big\{k\, \le\, j \,\le\, k+n \colon \,S_j = L \big\}}{n}}.
\end{equation}
Summoning \eqref{eq:fulldensity1B}, we conclude that for the positive integer $M$ one has
$$\inf_{k\,\in\, \N} \,\big|\omega_k(v_*^+) \dots \omega_{k+M}(v_*^+)\big|^{\frac1M} \,\,\geq\,\,\eta^{\,\inf_{k\,\in\, \N} \,\frac{\#\,\big\{k\, \le\, j \,\le\, k+M \colon \,S_j = L \big\}}{M}} \,\geq \,\,\eta^{\frac{1}{M}} \,> \,1.$$
Thus, by \cite[Lemma 19]{BernardesMessaoudiETDS2020}), the following limit exists and satisfies 
\begin{equation*}
\lim_{n\,\to\,+\infty} \,\,\inf_{k\,\in\, \N} \,\big|\omega_k(v_*^+) \dots \omega_{k+n}(v_*^+)\big|^{\frac1n} \,> \,1.
\end{equation*}
Analogous estimates show that
$$\lim_{n\,\to\,+\infty} \,\,\sup_{k\,\in\, \N} \,\big|\omega_k(v_*^-) \dots \omega_{k+n}(v_*^-)\big|^{\frac1n} \,<\, 1.$$
Therefore, 
item (C) of Corollary~\ref{cor:BM} implies that the shift operator $\sigma_{\mathcal S}$ satisfies the shadowing property. Moreover, as $\sigma_{\mathcal S}$ is conjugate to the product of two weighted backward shifts with the shadowing property, we have that it is generalized hyperbolic.}
\end{example}

\medskip

\subsubsection{Shift operators conjugate to skew-products of weighted shifts}\label{ex:notproduct}
We now present an example of a sequence $\mathcal S=(S_n)_{n\,\in\, \mathbb Z}$ without $\mathcal S-$bounded projections and whose shift operator is not linearly conjugate to the product of the two weighted backward shifts associated to the weights $\omega((1,0))$ and $\omega((0,1))$; it turns out to be conjugate to a skew-product of weighted backward shifts.

\begin{example}[$\mathcal S-$unbounded projections]\label{discontinuousprojections} 

\emph{Let $X=\mathbb R^2$ endowed with the norm given by $\|(x,y)\|=\max\,\{|x|,|y|\}$ and $\mathcal S=(S_n)_{n\,\in\, \Z}$ be the sequence of bounded linear operators expressed by the matrix 
\[ S_n \,=\, L \,= \,
 \begin{pmatrix}
    1 & 1  \\
    0 & 1  
 \end{pmatrix}
  \qquad \forall\, n\in \Z.
\]
A simple computation shows that, for every $(x_n,y_n)_{n\,\in\, \Z} \in \ell_p(\mathbb R^2)$ and any $k\in \Z$, one has
\begin{equation}\label{eq:iter}
\sigma_{\mathcal S}^k\,\big(\,(x_n,y_n)_{n\,\in\, \Z}\,)\big) \,
 = \,\big(x_{n+k}+k \,y_{n+k}, {\,y_{n+k})_{n\,\in\, \Z}}.
\end{equation}
Moreover, for every $n \in \N$, 
$$
e_{-n}(1,0)=\frac{L^n(1,0)}{\|L^n(1,0)\|}=(1,0)
    \quad\quad \text{and}\quad \quad 
e_{-n}(0,1)=\frac{L^n(0,1)}{\|L^n(0,1)\|}=\frac{(n,1)}{n}=\Big(1,\, \frac1n\Big).
$$
Hence, the angle between these two vectors goes to zero as $n\to+\infty$ and
\begin{equation*}
\big\|\Pi_{e_{-n}(0,1)}\big\| \, \ge\, \frac{\|\Pi_{e_{-n}(0,1)}\big((0,1)\big)\|}{\|(0,1)\|} \,=\, \big\|\Pi_{e_{-n}(0,1)}\big(-n\,e_{-n}(1,0) + n\, e_{-n}(0,1))\big\| \,=\,  n.    
\end{equation*}
Consequently, the canonical basis $\cE=\{(1,0),(0,1)\}$ does not have $\mathcal S-$bounded projections. A similar argument shows that every basis $\cE$ of $\mathbb R^2$ has no $\mathcal S-$bounded projections. } 

\smallskip


\emph{We now verify that the shift operator $\sigma_{\mathcal S} \colon \ell_p(\mathbb R^2) \to \ell_p(\mathbb R^2)$ is conjugate to the skew-product of two weighted backward shifts. By the choice of the norm in $\mathbb R^2$, the linear map $\phi\colon \ell_p(\mathbb R^2) \to \ell_p(\mathbb R) \times \ell_p(\mathbb R)$ given by 
$$\phi\big((x_n,y_n)_{n\,\in\,\Z}\big)\,=\, \big((x_n)_{n\,\in\,\Z}, \,(y_n)_{n\,\in\,\Z}\big)$$
is an isomorphism.
Moreover,
\begin{align*}
\phi \circ  \sigma_{\mathcal S} \circ \phi^{-1}\,
\big((x_n)_{n\,\in\,\Z}, \,\,(y_n)_{n\,\in\,\Z}\big) 
& = \,\big(\phi \circ  \sigma_{\mathcal S}\big) \,\big((x_n,y_n)_{n\,\in\,\Z}\big) \\
& =\, \phi\, \big((x_{n+1} + y_{n+1}, \,y_{n+1})_{n\,\in\, \Z}\big) \\
& =\, \big((x_{n+1} + y_{n+1})_{n\,\in\, \Z},\,(y_{n+1})_{n\,\in\, \Z}\big) \\
& =\, \big((x_{n+1})_{n\,\in\, \Z} + (y_{n+1})_{n\,\in\, \Z}, \,(y_{n+1})_{n\,\in\, \Z}\big) \\
& =\, \widehat \sigma_{\mathcal S} \,\big((x_n)_{n\,\in\,\Z}, \ ,(y_n)_{n\,\in\,\Z}\big) 
\end{align*}
where $\widehat \sigma_{\mathcal S}$ is the skew-product 
defined by 
$$
\begin{array}{rccc}
  \widehat \sigma_{\mathcal S}\colon  & \ell_p(\mathbb R) \times \ell_p(\mathbb R) & \quad \to \quad & \ell_p(\mathbb R) \times \ell_p(\mathbb R) \\
     & \big((x_n)_{n\,\in\,\Z},\, (y_n)_{n\,\in\,\Z}\big) & \quad \mapsto \quad &
     \Big(B_\omega\,\big((x_n)_{n\,\in\,\Z}\big) + B_\omega\,\big((y_n)_{n\,\in\,\Z}), \,B_\omega\, \big((y_n)_{n\,\in\,\Z}\big)\Big)
\end{array}
$$
and $B_\omega$ is the weighted shift given by
\begin{equation*}\label{eq:Bw}
B_\omega \colon \ell_p(\mathbb R)\to \ell_p(\mathbb R)
\end{equation*}
with weights $\omega = (\omega_n)_{n \, \in \, \Z}$ and   $\omega_n = 1$ for all $n\in \Z$.}

\smallskip

\emph{This example shows that Theorem~\ref{thm:productfactorBanach} is sharp, that is, despite the fact that $\mathcal E = \{(1,0),(0,1)\}$ is a basis for $\R^2$, $\sigma_{\mathcal S}$ fails to be linearly conjugate to $B_{\omega((1,0))} \times B_{\omega((0,1))}$. Indeed, an easy computation shows that the weights for $B_{\omega((1,0))}$ are
\[\omega_n = 1  \ \ \ \forall n \in \Z\] 
and the weights for $B_{\omega((0,1))}$ are
\[\omega_n = \begin{cases}
    1 & n=0,1\\
    \frac{n-1}{n} & n \ge 2\\
    \frac{|n|+1}{|n|} & n \le -1.
\end{cases}
\]
As the weights of $B_{\omega((0,1))}$ are bi-Lipschitz equivalent to the constant sequence 1, we have that $B_{\omega((0,1))}$ is linearly conjugate to $B_{\omega((1,0))}$. In particular, if we let $T = B_{\omega((1,0))} \times B_{\omega((0,1))}$, then all orbits of $T$ are bounded. However, if $v \,=\, (0,y_n)_{n\,\in\, \Z} \in \ell_p(\R^2)$ with some $y_{n_0}\neq 0$, from \eqref{eq:iter} we deduce that
$$\lim_{k \, \to \, +\infty}\,\|\sigma_{\mathcal S}^k(v)\| \,= \, +\infty,$$ 
implying that $\sigma_S$ is not linearly conjugate to $B_{\omega((1,0))} \times B_{\omega((0,1))}$.}
\end{example}

We note that the shift operator $\sigma_{\mathcal S}$ does not have the shadowing property and hence it is not generalized hyperbolic. 
This is so because the shadowing property is preserved by factors, and $B_{\omega((1,0))}$, a factor of $\sigma_{\mathcal S}$, does not have the shadowing property as all of its weights are 1's.


\section{Proofs}\label{se:proofs}

\subsection{Dissipative operators on Banach spaces are shift operators}

This subsection is devoted to the proof of Theorem~\ref{thm:dissipative2shifts}. Let $(E_n)_{n\,\in\, \mathbb Z}$ be a {dissipative decomposition of $X$}, where  $E_0\subset X$ is a closed subspace, $E_n = T^n(E_0)$ for every $n\in \mathbb Z$ and
$$X = \overline{\bigoplus_{n\,\in\, \mathbb Z} E_n}.$$

Given $x\in X$, there exist unique vectors $x_n\in E_n$, for $n\in \mathbb Z$, such that $x = \sum_{n\,\in\, \mathbb Z} x_n$. Let $H\colon X \to {E_0}^{\mathbb Z}$ be defined by
$$H(x) = (T^{n}(x_{-n}))_{n\,\in\, \mathbb Z}.$$ 
By the uniqueness of the decomposition and the linearity of $T$, the map $H$ is well defined, linear and injective. We endow the vector space ${\bf B} = H(X)$ with the norm $\|\cdot\|_{\bf B}$ given, for each $(y_n)_{n\,\in\, \mathbb Z}\in {\bf B}$, by
$$\|(y_n)_{n\,\in\, \mathbb Z}\|_{\bf B} \,=\, \|x\|_X$$
where $H(x) = (y_n)_{n\,\in \,\mathbb Z}$. Therefore, $H$ is continuous and ${\bf B}$ is a Banach space. Thus, by the Open mapping theorem, $H$ is open.

\medskip

Now consider the sequence of linear operators $\mathcal S=(S_n)_{n\,\in\, \mathbb Z}$ acting on $E_0$ so that $S_n$ is the identity for each $n\in \mathbb Z$. We claim that 
$$
\sigma_{\mathcal S} \circ H = H\circ T 
$$
where 
$\sigma_{\mathcal S}\colon {\bf B} \to {\bf B}$ is the shift operator determined by $\mathcal S$. Indeed, let $x\in X$ given by $x=\sum_{n\,\in\, \mathbb Z} \,x_n$, with $x_n\in E_n$. Then $T(x) = \sum_{n\,\in\, \mathbb Z} T(x_n)$, and so 
$$H(T(x)) \,=\, \big(T^n(T(x_{-n-1}))\big)_{n\,\in\, \mathbb Z}\,=\, \big(T^{n+1}(x_{-(n+1)})\big)_{n\,\in\, \mathbb Z}.$$
Analogously, 
$$
\sigma_{\mathcal S}(H(x))
   \, = \,\sigma_{\mathcal S} \big((T^{n}(x_{-n})_{n\,\in\, \mathbb Z}\big)
   \, =\, \big(T^{n+1}(x_{-(n+1)})\big)_{n\,\in\, \mathbb Z}.
$$
This completes the proof of the theorem.

\subsection{Dissipative composition operators are shift operators}

In this subsection we will prove Theorem~\ref{thm:shift-like2shift}.
Assume $(M,{\mathfrak{B}},\mu, f)$  and $T_f$ are as stated in Theorem~\ref{thm:shift-like2shift}. We need to find a suitable Banach space $Y$, a sequence $\mathcal{S}= (S_n)_{n \, \in \, \Z} $ of invertible linear operators  on $Y$ and a Banach space ${\bf B} \subset Y^\mathbb Z$, such that $ ({\bf B}, \sigma_ {\mathcal S})$ is conjugate to $T_f$.
\smallskip

To this end, let $Y = L^p (W, \mu)$ and consider the vector space  
$$
{\bf B} \,=\, \Big\{ (\psi_n)_{n\,\in \,\mathbb Z} \in Y^{\mathbb Z} \colon \sum_{n\,\in\, \mathbb Z}\, \int_W |\psi_n|^p \,\frac{d\mu f^{n}}{d\mu} \, d\mu < +\infty \Big\}
$$ 
endowed with the norm 
$$
\|(\psi_n)_{n\,\in \,\mathbb Z}\|_{\bf B} \,=\, \left(\sum_{n\,\in\, \mathbb Z} \, \int_W |\psi_n|^p \,  \frac{d\mu f^{n}}{d\mu}\, d\mu\right)^{\frac{1}{p}}.
$$
Next we define a map $\Gamma\colon L^p(M, \mu) \rightarrow \bf B$ as follows: for each $\varphi \in L^p(M, \mu)$, let $\psi = \Gamma (\varphi) \in \bf B$ be given by $\psi_n = \varphi \circ f^n$. Note that the domain of $\psi_n$ is $W$. 
\smallskip

We begin by showing that $\Gamma$ is well defined and $1-$Lipschitz: 
\begin{align*}
    \| \Gamma (\varphi) \|_{\bf B} ^p 
     = 
     \| \psi \|_{\bf B}^p
     & =  \sum _{n \,\in\, \Z} \,\int_W |\psi_n|^p \,
     \frac{d\mu f^{n}}{d\mu} d\mu\\
     & = \sum _{n \,\in\, \Z} \, \int _W | \psi_n|^p \,{d\mu f^{n}}  \\
    & = \sum _{n \,\in\, \Z} \, \int _{f^n(W)} |\psi_n|^p\circ f^{-n} \, d\mu \\
& = \sum _{n \,\in\, \Z} \, \int _{f^n(W)} |\varphi|^p \, d\mu \\
    & =  \| \varphi\|_{L^p(M, \mu)}.
\end{align*}
Afterwards, we show that $\Gamma$ maps onto $\bf B$ and its inverse is also $1-$Lipschitz. Indeed, let $\psi  = (\psi_n)_{n \,\in\, \Z}  \in \bf B$. We define  $\varphi\colon M \rightarrow M $ by $\varphi (x) = \psi_n(f^{-n}x)$, where $n \in \Z$ is the unique integer such that $x \in f^n(W)$. Then
    \begin{align*}
    \|\varphi \|_{L^p(M)}^p & = \sum _{n \,\in\, \Z}  \int _{f^n(W)} | \varphi|^p \, d\mu\\  
    & =\sum _{n \,\in\, \Z}  \int _{f^n(W)} | \psi_n |^pf^{-n}  \,d\mu\\
    & = \sum _{n \,\in\, \Z}  \int _{W} | \psi_n|^p \,{d \mu f^{n}} \\
& = \sum _{n \,\in\, \Z}  \int _{W} | \psi_n|^p \,\frac{d \mu f^{n}}{d \mu} \, d\mu \\
    &=  \|\psi \|_{\bf B}^p
\end{align*}
confirming that $\Gamma$ is surjective and its inverse is $1-$Lipschitz as well. Consequently, $\Gamma$ is an isometry and $\bf B$ is a Banach space. 
\smallskip

For each $n \in \Z$, let  $S_n\colon Y \rightarrow Y$ be the identity map and $\cS = (S_n)_{n \, \in \, \Z}$. We next show that $\Gamma$ is a conjugacy between $T_f$ and $\sigma_{\cS}$.  Indeed, given $\varphi \in L^p(M, \mu)$, one has
$$\Gamma (T_f(\varphi)) \, =\, \Gamma (\varphi \circ f) 
\, =\,\big((\varphi \circ f)f^n|_W\big)_{n\, \in\, Z} \,=\, \big(\varphi \circ f^{n+1}|_W\big)_{n\, \in\, \Z}$$
and 
$$\sigma_{\cS} \big(\Gamma (\varphi)\big) \, =\, \sigma_{\cS} \big((\varphi \circ f^n|_W)_{n \,\in\, \Z} \big) \, = \,\big(\varphi \circ f^{n+1}|_W\big)_{n \,\in \,Z}.$$
The proof of the theorem is complete.

\begin{remark}\label{rem:RN}
Whenever there is $C >0 $ such that 
\begin{equation}\label{eq:RNK}
\frac{1}{C}  \,<\, \big\|\frac{d\mu f^n}{d \mu}\big\|_{\infty} \,<\, C \quad \quad \forall\, n \in \Z
\end{equation}
the Banach space $\bf B$ provided by Theorem~\ref{thm:shift-like2shift} is isomorphic to $\ell_p( L^p(\mu\mid_W))$ and $\sigma_{\mathcal S}$ is conjugate to the usual shift map on $\ell_p( L^p(\mu\mid_W))$. 
\end{remark}


\subsection{Weighted shifts as factors of shift operators}

Here we prove Theorem~\ref{thm:factors}. Let $\mathcal E$ be an orthogonal basis for $\mathcal S=(S_n)_{n\in Z}$.
Given $x\in \mathcal E$ recall that 
\begin{equation*}
e_n(x) \,=\,
\begin{cases}
\begin{array}{cc}
\frac{(\,S_{[1,n]}\,)^{-1}\,(x)}{\|(\,S_{[1,n]}\,)^{-1}\,(x)\|}  & \quad \text{if} \,\, n \geq 0 \medskip \\
\frac{S_{[n+1,0]}\,(x)}{\|S_{[n+1,0]}\,(x)\|}  & \quad  \text{if} \,\, n <0
\end{array}.
\end{cases}
\end{equation*}
\medskip
and that, by construction of the weights, one has 
$$S_{n+1} \,\big(e_{n+1}(x)\big)\,=\, \omega_{n+1}(x) \,e_n(x) \,  \quad \quad \forall\, n \in Z.$$
Consider the function 
$\Gamma=\Gamma_x\colon \ell_p(X) \to \mathbb \ell_p(\mathbb K)$ given by 
\begin{equation}\label{eq:defGamma}
    \Gamma\,\big((x_n)_{n\,\in\, \mathbb Z}\big)\,=\,\big(\Pi_{e_n(x)} \,(x_n)\big)_{n \, \in \, \Z}
\end{equation}
where $\Pi_{e_n(x)}\colon X \to F_n(x)$ stands for the orthogonal projection from $X$ onto the subspace $F_n(x)$ generated by the unit vector $e_n(x)$. We claim that the linear map $\Gamma\colon \ell_p(X) \to \ell_p(\mathbb K)$ is continuous and surjective. 

As $X$ is a Hilbert space, all orthogonal projections $\Pi_z$ are uniformly bounded (with norm one), and so 
\begin{align*}
    \|\Gamma\big((x_n)_{n\,\in \,\Z}\big)\|_p^p 
    \,=\, \sum_{n\,\in\, \Z} \,\|\Pi_{e_n(x)}\,(x_n)\|^p
    \,\le\, \sum_{n\,\in\, \Z} \,\|x_n\|^p
    \,=\, \|(x_n)_{n\,\in \,\Z})\|_p^p.
\end{align*}
This proves the continuity of $\Gamma$ and also that $\Gamma\big((x_n)_{n\,\in \,\Z}\big)$ belongs to $\ell_p(\mathbb K)$.

To show the surjectivity of $\Gamma$ we just observe that, given $(t_n)_{n\,\in\, \mathbb Z}\in \ell_p( \mathbb K)$, if we consider $u_n = t_n e_n(x)$ for every $n\in \Z$, then 
$$\Gamma\big((u_n)_{n\,\in\, \mathbb Z}\big) \,=\, \Gamma\big((t_n e_n(x))_{n\,\in\, \mathbb Z}\big) \,=\,  \big(t_n\,\Pi_{e_n(x)} \,(e_n(x))\big)_{n\,\in\, \mathbb Z}\,=\,(t_n)_{n \, \in \,\Z}.$$
Moreover, since all vectors $e_n$ are unitary, it is clear that $\|(u_n)_{n\,\in \,\mathbb Z}\|_p = \|(t_n)_{n\,\in\, \mathbb Z}\|_p$.
\smallskip

We are left to prove that $\Gamma\circ \sigma_{\mathcal S}=B_{\omega(x)}\circ \Gamma$. On the one hand, given $(x_n)_{n\,\in\, \mathbb Z} \in \ell_p( X)$, 
\begin{eqnarray*}
\Gamma \,\big(\sigma_{\mathcal S}\,((x_n)_{n\,\in\, \mathbb Z})\big) &=&\Gamma\,\big((S_{n+1}(x_{n+1}))_{n\,\in\, \mathbb Z}\big) \\
&=&  \big(\Pi_{e_n(x)} \big(S_{n+1}(x_{n+1})\big)\big)_{n\,\in\, \mathbb Z}\\
&=& \big(\Pi_{\frac{1}{\omega_{n+1}}\,S_{n+1} \,(e_{n+1}(x))} \big(S_{n+1}(x_{n+1})\big)\big)_{n\,\in\, \mathbb Z}
\end{eqnarray*}
and
\begin{eqnarray*}
B_{\omega(x)}\, \big(\Gamma \,((x_n)_{n\,\in\, \mathbb Z})\big) &=& B_{\omega(x)}\,\big((\Pi_{e_n(x)} (x_n))_{n\,\in\, \mathbb Z}\big) \\
&=& \big(\omega_{n+1} \, \Pi_{e_{n+1}(x)} (x_{n+1})\big)_{n\,\in\, \mathbb Z}.
\end{eqnarray*}
On the other hand, we may write $x_{n+1}$ as a sum
$$x_{n+1} \,= \, \lambda_{n+1} \, e_{n+1}(x) + \gamma_{n+1}$$
where $\lambda_{n+1} \in \R$ and the vector $\gamma_{n+1} \in X$ is orthogonal to $e_{n+1}(x)$. Moreover, for each $b\in \cE\setminus \{x\}$, there exists $\eta_b\in \mathbb K$ such that 
$$
\gamma_{n+1}\, =\, \sum_{b\,\in\, \cE,\; b\,\neq\, x} \; \eta_{b} \; e_{n+1}(b).
$$
Therefore,
\begin{align*}
S_{n+1} \,(x_{n+1}) \,& = \, \lambda_{n+1} \, S_{n+1}\, (e_{n+1}(x)) + 
\sum_{b\,\in \,\cE,\; b\,\neq\, x} \; \eta_{b} \; S_{n+1}\, (e_{n+1}(b))
\end{align*}
where 
$$\big\{S_{n+1}\big(e_{n+1}(b)\big)\colon  \,\,b \in \cE\big\} \,= \,\big\{\omega_{n+1}(b)\,e_{n}(b)\colon \,\,b \in \cE\big\}$$
is a set of pairwise orthogonal vectors by the assumption on the sequence $\mathcal S$. Thus,
$$\Pi_{\frac{1}{\omega_{n+1}(x)}\,S_{n+1} \,(e_{n+1}(x))} \big(S_{n+1}\,(x_{n+1})\big) \, = \, \omega_{n+1}(x)\, \lambda_{n+1} \,=\, \omega_{n+1}(x) \, \Pi_{e_{n+1}(x)} (x_{n+1}).$$
This ends the proof of the theorem.


\subsection{Classification theorems}\label{sse:class}

\subsubsection{Proof of Theorem~\ref{thm:productfactor}}

The statement of the theorem is immediate if the Hilbert space $X$ has  dimension $d = 1$, since in this case the map $\sigma_{\mathcal S}$ is, by definition, a weighted backward shift whose weights are precisely $(\omega_n)_{n \, \in \, \Z}$.
\smallskip

Let $X$ be a Hilbert space with dimension $d \geq 2$ and $\cE = \{b_1, \ldots, b_d\}$ be an orthogonal basis for $(S_n)_{n\, \in\, \Z}$, that is, $\cE_n = \{e_n(x) \colon x\in \cE\}$ is an orthonormal basis on $X$ for every $n \in \Z$. Endow the vector space $\big(\ell_p(\mathbb K)\big)^d$ with the sum norm. Consider the map
$$I\colon \ell_p(X) \,\,\to\,\, \big(\ell_p(\mathbb K)\big)^d$$
given by
\begin{eqnarray*}
I\big((x_n)_{n\,\in\, \Z}\big)
&=&\left(\,
\Pi_{e_n(b_1)}(x_n),
\Pi_{e_n(b_2)}(x_n),
\dots, 
\Pi_{e_n(b_d)}(x_n)\,
\right)_{n\,\in\, \Z}\\
&=& \left(\,
\Gamma_{b_1}\big((x_n)_{n \, \in \, \Z}\big),
\Gamma_{b_2}\big((x_n)_{n \, \in \, \Z}\big),
\dots, 
\Gamma_{b_d}\big((x_n)_{n \, \in \, \Z}\big)\,
\right)_{n\,\in\, \Z}.
\end{eqnarray*}
It is clear that $I$ is linear and injective. 
Moreover, $I$ is surjective and continuous. Indeed, given 
$$\Big(\big(\alpha_{n,1}\big)_{n\,\in\, \Z},\,\big(\alpha_{n,2}\big)_{n\,\in\, \Z},\, \dots,\, \big(\alpha_{n,d}\big)_{n\,\in \Z}\Big) \,\,\in\,\, (\ell_p(\mathbb K))^d$$
let 
$$x_n \,=\,\sum_{i=1}^d \,\alpha_{n,i}\,\,e_n(b_i) \,\, \in X$$
for all $n\in \Z$. By definition, 
$$I((x_n)_{n\,\in\, \Z})=\big((\alpha_{n,1})_{n\,\in\, \Z},(\alpha_{n,2})_{n\,\in\, \Z}, \dots, (\alpha_{n,d})_{n\,\in\, \Z}\big).$$
Moreover, $(x_n)_{n \, \in \, \Z}$ belongs to $ \ell_p(X)$ since
\begin{align*}
\|x_n\|^p
    & \,= \, \Big\|\sum_{i=1}^d \,\alpha_{n,i} \,\,e_n(b_i) \Big\|_X^p \\
    & \,\leq \, \Big(\sum_{i=1}^d \,\Big\|\alpha_{n,i} \,\,e_n(b_i) \Big\|_X\Big)^p \\
    & \,\leq \,\,d^{p-1} \sum_{i=1}^d \,
    \Big\|\alpha_{n,i} \,\,e_n(b_i) \Big\|_X^p \\
    & \,= \,\,d^{p-1} \sum_{i=1}^d \,
    |\alpha_{n,i}|^p
\end{align*}
where the third inequality is due to the fact that, as $p \geq 1$, the map $t \geq 0 \mapsto t^p$ is convex, and in the last equality we use the assumption that $\cE_n$ is a normalized basis. Therefore
\begin{align*}
\sum_{n\,\in \,\Z} \,\|x_n\|^p
    & \,\leq \,\sum_{n\,\in\, \Z} \,d^{p-1} \,\sum_{i=1}^d \, |\alpha_{n,i}|^p \\
    & \,= \, d^{p-1}\,\sum_{i=1}^d  \,\sum_{n\,\in\, \Z}  \,    |\alpha_{n,i}|^{p}  \, < +\infty.
\end{align*}
By the Open mapping theorem, $I$ is an isomorphism. Furthermore, by construction,
\begin{align*}
    (I\circ \sigma_{\mathcal S})\big((x_n)_{n\,\in\, \Z}\big)
    & = I\left(\big(S_{n+1}(x_{n+1})\big)_{n\,\in\, \Z}\right) \\
    & = \left(\Gamma_{b_1}\big((S_{n+1}(x_{n+1}))_{n\,\in\, Z}\big),\,
    \dots, \,
    \Gamma_{b_d}\big((S_{n+1}(x_{n+1}))_{n\,\in\, Z}\big) \right)
    \\ &
    =   \left( B_{\omega(b_1)}\big(\Gamma_{b_1}((x_n)_{n\,\in\, \Z})\big),\,
    \dots, \,
    B_{\omega(b_d)}\big(\Gamma_{b_d}((x_n)_{n\,\in\, \Z}) \big)\right)
    \\
    & = \left(B_{\omega(b_1)} \times B_{\omega(b_2)} \times \dots \times B_{\omega(b_d)}\right)
    \big(I((x_n)_{n\,\in\, Z})\big).
\end{align*}
This completes the proof of the theorem.

\subsubsection{Proof of Theorem~\ref{thm:productfactorBanach}}
Fix $1\leq p < +\infty$ and let $\cE = \{b_1, \ldots, b_d\}$ be a basis of $X$ with $\mathcal S-$bounded projections. For each $n\in \Z$, consider the basis $\cE_n=\{e_n(x) \colon x\in \cE\}$. Given $b\in \cE$, take the linear map 
$$
\begin{array}{rccc}
 \Gamma_b \colon  
 & \ell_p(X) & \quad \to \quad & \mathbb K^{\mathbb Z} \qquad  \\
& (x_n)_{n\,\in\, \mathbb Z} & \quad \mapsto \quad & 
     \big(\Pi_{e_n(b)} \,(x_n)\big)_{n \, \in \, \Z}.
\end{array}
$$
We claim that $\Gamma_b$ is continuous and $\Gamma_b(\ell_p(X))=\ell_p(\mathbb K)$. Indeed, by the $\mathcal S-$bounded projections assumption \eqref{eq:boundednorms}, there exists $C>0$ such that 
\begin{align*}
    \|\Gamma_b\big((x_n)_{n\,\in \,\Z}\big)\|_p^p 
    \,=\, \sum_{n\,\in\, \Z} \,\|\Pi_{e_n(b)} \,(x_n)\|^p
    \,\le\, C^p\, \sum_{n\,\in\, \Z} \,\|x_n\|^p
    \,=\, C^p\, \|(x_n)_{n\,\in \,\Z})\|_p^p.
\end{align*}
This estimate proves the continuity of $\Gamma_b$ and also that $\Gamma_b\big((x_n)_{n\,\in \,\Z}\big) \in \ell_p(\K)$. Conversely, given $(t_n)_{n\,\in\, \mathbb Z}\in \ell_p(\mathbb K)$, the vector $\big(t_n e_n(b)\big)_{n\,\in\, \mathbb Z}$ belongs to $\ell_p(\mathbb K)$ because the elements in $\cE_n$ are normalized, and one has 
$$\Gamma_b\big(\big(t_n e_n(b)\big)_{n\,\in\, \mathbb Z}\big) \, = \, \big(t_n\big)_{n\,\in\, \mathbb Z}.$$

In order to construct the claimed conjugation we will need the following auxiliary result.

\begin{lemma}\label{le:controlnormsBanach}
 Let $X$ be a finite dimensional Banach space, 
$\mathcal S=(S_n)_{n\,\in\, \mathbb Z}$ be a sequence of operators on $X$ and $\cE$ be a basis of $X$ with $\mathcal S-$bounded projections. There exists $\mathcal K_p>0$ such that, for every $(\alpha_b)_{b\,\in\, \cE} \in \mathbb K^{\dim X}$ and $n\in \Z$, one has
$$
\Big\| \sum_{b\,\in\, \cE} \,\alpha_b \; e_n(b) \Big\|^p  \,\le\, \mathcal K_p\,
\sum_{b\,\in\, \cE} \, \left\| \alpha_b \; e_n(b) \right\|^p \,=\, \mathcal K_p \, \sum_{b\,\in\, \cE}\,  |\alpha_b|^p.$$
\end{lemma}

\begin{proof} Consider $1 < p < +\infty$ and take $q \in \N$ such that $\frac{1}{p} + \frac{1}{q} = 1$. Given $x \in X$, write 
$$x \,=\, \sum_{b\,\in\, \cE} \,\alpha_b \; e_n(b) \, = \, \sum_{b\,\in\, \cE} \, \Pi_{e_n(b)} (x).$$    
Then, using the triangular inequality, the assumption \eqref{eq:boundednorms} and H\"older inequality, we get
\begin{align*}
\|x\|^p 
&\, = \,\Big\|\sum_{b\,\in\, \cE} \, \Pi_{e_n(b)}(x)\Big\|^p 
\,\le\, \left(\sum_{b\,\in\, \cE} \,C\, \left\| \, \alpha_b \,e_n(b)\right\|\right)^p  \\
& \,=\, C^p\, \left(\sum_{b\,\in\, \cE} \, |\alpha_b|\cdot 1 \right)^p \\
&\,\le\, C^p \, \left[\left(\sum_{b\,\in\, \cE} \, |\alpha_b|^p\right)^\frac1p \,\left(\sum_{b\,\in\, \cE} \, 1^q\right)^\frac1q \right]^p \\
&\,\le\, C^p \, \big(\dim X\big)^\frac{p}{q} \; \sum_{b\,\in\, \cE} \, |\alpha_b|^p \\
& \,=\, C^p\, \big(\dim X\big)^{p\,(p-1)}\; \sum_{b\,\in \,\cE} \, |\alpha_b|^p.
\end{align*}

In case $p=1$, given $x \in X$ written as
$$x \,=\, \sum_{b\,\in\, \cE} \,\alpha_b \; e_n(b) \, = \, \sum_{b\,\in\, \cE} \, \Pi_{e_n(b)} (x)$$ 
then, using the triangular inequality and the assumption \eqref{eq:boundednorms}, we get
$$\|x\| \, = \,\Big\|\sum_{b\,\in\, \cE} \, \Pi_{e_n(b)}(x)\Big\| \,\le\,\sum_{b\,\in\, \cE} \,C\, \left\| \, \alpha_b \,e_n(b)\right\| \,=\, C\, \sum_{b\,\in\, \cE} \, |\alpha_b|.$$
\end{proof}

\medskip
Consider the map $I\colon \ell_p(X) \to (\ell_p(\mathbb K))^d$ defined by
\begin{eqnarray*}
I\big((x_n)_{n\,\in\, \Z}\big)
&=&\left(\big(\Pi_{e_n(b_1)}(x_n)\big)_{n\,\in\, \Z}, \,
\big(\Pi_{e_n(b_2)}(x_n)\big)_{n\,\in\, \Z},\,
\dots, \,\big(\Pi_{e_n(b_d)}(x_n)\big)_{n\,\in\, \Z}\,
\right)\\
&=& \left(\Gamma_{b_1}\big((x_n)_{n \, \in \, \Z}\big),\,
\Gamma_{b_2}\big((x_n)_{n \, \in \, \Z}\big),\,
\dots, \,
\Gamma_{b_d}\big((x_n)_{n \, \in \, \Z}\big)
\right)_{n\,\in\, \Z}
\end{eqnarray*}
By construction, $I$ is linear, injective and continuous. Moreover, $I$ is surjective since, given 
$$\Big((\alpha_{n,1})_{n\,\in\, \Z},(\alpha_{n,2})_{n\,\in\, \Z}, \dots, (\alpha_{n,d})_{n\,\in\, \Z}\Big) \,\in\, (\ell_p(\mathbb K))^d$$
one has 
$$I\Big(\Big(\sum_{i=1}^d \alpha_{n,i} \,e_n(b_i)\Big)_{n\,\in\,\Z}\Big) \,=\,\Big((\alpha_{n,1})_{n\,\in\, \Z},(\alpha_{n,2})_{n\,\in\, \Z}, \dots, (\alpha_{n,d})_{n\,\in\, \Z}\Big)
$$
and, by using Lemma~\ref{le:controlnormsBanach}, 
$$ \sum_{n\,\in\, \Z}\, \Big\|\sum_{i=1}^d \,\alpha_{n,i} \,e_n(b_i) \Big\|^p \,\le\,  \sum_{n\,\in\, \Z} \, \mathcal K_p\, \sum_{i=1}^d \, |\alpha_{n,i}|^p 
 \,=\, \mathcal K_p \sum_{i=1}^d  \,\sum_{n\,\in\, \Z}  \,
    |\alpha_{n,i}|^p \, <\,+\infty.$$
Therefore, $I$ is an isomorphism. Furthermore, by construction,
\begin{align*}
    (I\circ \sigma_{\mathcal S})\big((x_n)_{n\,\in\, \Z}\big)
    & = I\left(\big(S_{n+1}(x_{n+1})\big)_{n\,\in\, \Z}\right) \\
    & = \left(\Gamma_{b_1}\big((S_{n+1}(x_{n+1}))_{n\,\in\, Z}\big),\, \dots, \,
    \Gamma_{b_d}\big((S_{n+1}(x_{n+1}))_{n\,\in\, Z}\big) \right)
    \\ &
    =   \left( B_{\omega(b_1)}\big(\Gamma_{b_1}((x_n)_{n\,\in\, \Z})\big),\,  \dots, \,
    B_{\omega(b_d)}\big(\Gamma_{b_d}((x_n)_{n\,\in\, \Z}) \big)\right)
    \\
    & = \left(B_{\omega(b_1)} \times B_{\omega(b_2)} \times \dots \times B_{\omega(b_d)}\right)
    \big(I((x_n)_{n\,\in\, Z})\big).
\end{align*}
This finishes the proof of the theorem.

\subsubsection{Proof of Corollary~\ref{cor:separatedbasis-2}}

Let $X$ be a finite dimensional Hilbert space with dimension $d \geq 2$ and assume that the sequence $\mathcal S=(S_n)_{n\,\in \,\Z}$ has a normalized basis $\cE = \{v_1, \cdots, v_d\}$ of $X$ for which there exists $0 < \gamma < 1/(d-1)$ such that, for every $n \in \Z$, the angle $\measuredangle(u, v)$ between distinct vectors $u,v \in \cE_n$ satisfies the condition
$$ \cos \measuredangle(u, v) \,\in\, [-\gamma, \gamma].$$ 

Consider $x \in X \setminus\{0\}$ whose coordinates with respect to the basis $\cE_n$ are
$$x \,= \,\alpha_{n,1} \,e_n(v_1) + \cdots + \alpha_{n,d} \,e_n(v_d)$$
where $\alpha_{n,i} \in \K$ for every $1 \leq i \leq d$. Then 
\begin{eqnarray*}
\|x\|_X^2 &=& \|\alpha_{n,1}\, e_n(v_1) + \cdots + \alpha_{n,d} \,e_n(v_d)\|_X^2  \\
&=& \|\alpha_{n,1}\|^2 + \cdots + \|\alpha_{n,d}\|^2 \,\,+ \sum_{1 \,\leq\, i, \,  j\, \leq \,d} \,\alpha_{n,i}\,\,\overline{\alpha_{n,j}} \,<e_n(v_i), e_n(v_j)> \\
&=& \|\alpha_{n,1}\|^2 + \cdots + \|\alpha_{n,d}\|^2 \,\,+ \sum_{1 \,\leq\, i\,< \,  j\, \leq \,d} \,2\, \mathfrak{Re}\big(\alpha_{n,i}\,\,\overline{\alpha_{n,j}}\big) \,\cos \theta_{n,i,j} \\
&\geq& \|\alpha_{n,1}\|^2 + \cdots + \|\alpha_{n,d}\|^2 \,\,- \sum_{1 \,\leq\, i\,< \,  j\, \leq \,d} \,2\, |\mathfrak{Re}\big(\alpha_{n,i}\,\,\overline{\alpha_{n,j}}\big)| \,|\cos \theta_{n,i,j}| \\
&\geq& \|\alpha_{n,1}\|^2 + \cdots + \|\alpha_{n,d}\|^2 \,\,- \sum_{1 \,\leq\, i\,< \,  j\, \leq \,d} \,2\, \|\alpha_{n,i}\,\,\overline{\alpha_{n,j}}\| \,|\cos \theta_{n,i,j}| \\
&=&  \|\alpha_{n,1}\|^2 + \cdots + \|\alpha_{n,d}\|^2 \,\,- \sum_{1 \,\leq\, i\,< \,  j\, \leq \,d} \,2\, \|\alpha_{n,i}\|\,\|\alpha_{n,j}\| \,|\cos \theta_{n,i,j}| \\
&\geq& \|\alpha_{n,1}\|^2 + \cdots + \|\alpha_{n,d}\|^2 \,\,- \sum_{1 \,\leq\, i\,< \,  j\, \leq \,d} \, 2\,\|\alpha_{n,i}\| \,\|\alpha_{n,j}\|\, \gamma \\
&\geq& \|\alpha_{n,1}\|^2 + \cdots + \|\alpha_{n,d}\|^2 \,\,- \sum_{1 \,\leq\, i\, <\, j\, \leq \,d} \, \big(\|\alpha_{n,i}\|^2 + \|\alpha_{n,j}\|^2\big)\, \gamma \\
&\geq& \sum_{i=1}^d \,\|\alpha_{n,i}\|^2 \,\Big(1 - \sum_{j \neq i} \gamma\Big) \\
&=&  \Big(\|\alpha_{n,1}\|^2 + \cdots + \|\alpha_{n,d}\|^2\Big) \, \Big(1 - \big(d-1\big)\,\gamma\Big).
\end{eqnarray*}

\noindent Observe now that, by the choice of $\gamma$, one has $1 - \big(d-1\big)\,\gamma > 0.$ Consequently, for every $n \in \Z$ and $j \in \{1,\cdots, d\}$, 
\begin{eqnarray*}
\sup_{x \,\neq \,0} \,\frac{\|\Pi_{e_n(v_j)}(x)\|_X}{\|x\|_X} &=& \frac{\|\alpha_{n,j}\|}{\|x\|_X} \\
&\leq& \frac{\|\alpha_{n,j}\|}{\sqrt{\Big(1 - \big(d-1\big)\,\gamma\Big)  \,\Big(\|\alpha_{n,1}\|^2 + \cdots + \|\alpha_{n,d}\|^2\Big)}} \\
&\leq& \frac{1}{\sqrt{\Big(1 - \big(d-1\big)\,\gamma\Big)}}.
\end{eqnarray*}
That is, the basis $\cE$ has $\mathcal S-$bounded projections. Therefore, by Theorem~\ref{thm:productfactorBanach}, for every $1 \le p < +\infty$ the shift operator $\sigma_{\mathcal S}\colon \ell_p(X) \to \ell_p(X)$ is conjugate to the product of weighted backward shifts $\prod _{x \,\in\, \cE} B_{\omega(x)} \colon \big(\ell_p(\mathbb K)\big)^{\dim X} \to \big(\ell_p(\mathbb K)\big)^{\dim X}$. The proof of the corollary is complete.

\bigskip

We note that, when $d=2$, the demand on the angles 
$\big\{\theta_{n,i,j}\colon n \in \Z,\, 1 \leq i \neq j \leq 2\big\}$ in the previous proof reduces to requesting that they do not approach neither $0$ nor $\pi$. The next example shows that, when $d > 2$, this is no longer enough to guarantee that the basis $\cE$ has $\mathcal S-$bounded projections and we need a sharper uniform bound. 

\begin{example}[Unbounded projections]\label{ex:class}
\emph{Assume that $d = 3$ and $\mathbb K = \mathbb R$. Given a suf\-ficiently small $\delta >0$, we may choose a normalized basis $\cE_\delta = \{e_1, e_2, e_3\}$ of $\R^3$ with angles $\theta_{i,j}$ sufficiently close to $2\pi/3$ so that $\cos \,(\theta_{1,2}) = -1/2$, 
$\cos\, (\theta_{1,3}) = \cos\, (\theta_{2,3}) =-1/2 + \delta$. If  $x = \alpha_1\,e_1 + \alpha_2\, e_2 + \alpha_3 \, e_3$, then 
\begin{eqnarray*}
\|x\|^2 &=& \|\alpha_1\, e_1 + \alpha_2\, e_2 + \alpha_3 \,e_3\|^2  \\
&=& \alpha_1^2 + \alpha_2^2 + \alpha_3^2  + 2\, \alpha_1\,\alpha_2 \cos\,(\theta_{1,2})\, + 2\, \alpha_1\,\alpha_3 \cos\,(\theta_{1,3}) + 2\, \alpha_2\,\alpha_3 \cos\,(\theta_{2,3}) \\
&=& \alpha_1 \big(\alpha_1 + 2\alpha_2 \cos\,(\theta_{1,2}) \big) \,+\, \alpha_2 \big(\alpha_2 \,+\,  2 \alpha_3 \cos\,(\theta_{2,3}) \big) \, + \, \alpha_3 \big(\alpha_3  \, + 2 \,\alpha_1\cos\,(\theta_{1,3})\big).
\end{eqnarray*}
In the particular case of $\alpha_1 =\alpha_{2} = \alpha_{3} = 1$ and $x_0 = e_1 + e_2 + e_3$, we get
$$\|x_0\|^2 \,=\, 0 + 1 (2 \delta) + 1 (2\delta) \,\, = 4 \delta\,\,$$
and so, for $1 \leq j \leq 3$,
$$\sup_{x \,\neq \,0} \,\frac{\|\Pi_{e_j}(x)\|}{\|x\|} \,\geq\, \frac{\|\Pi_{e_j}(x_0)\|}{\|x_0\|}  \, = \, \frac{|\alpha_j|}{\|x_0\|} \,=\, \frac{1}{4 \delta}.$$
As $\delta$ may be chosen arbitrarily small, the projection $\Pi_{e_j}$ is not uniformly bounded independent of $\delta$. We note that, as $\delta$ goes to $0$, the space generated by $\{e_1, e_2, e_3\}$ becomes 2-dimensional, so the angle between $e_3$ and the plane generated by 
$\{e_1, e_2\}$ approaches $0$. }
\end{example}


\subsubsection{Proof of Corollary~\ref{cor:angles}}

The next proposition provides another sufficient condition for a sequence $\mathcal S$ of bounded linear invertible operators to have $\mathcal S-$bounded projections. Corollary~\ref{cor:angles} is an immediate consequence of this proposition and  Theorem~\ref{thm:productfactorBanach}.

\begin{proposition}\label{prop:angles}
Let $X$ be a finite dimensional Hilbert space, $\cE=\{v_1,v_2, \dots, v_d\}$ be a basis of unit vectors of $X$ and $\mathcal S = (S_n)_{n\,\in\, \Z}$ be a sequence of bounded linear invertible operators on $X$. For each $1\le i \le d$, consider the vector space $F_{n,i} \subset X$ generated by the vectors $\big\{e_n(v_j) \colon 1\le j \le d, \; j\neq i\big\}$. If
\begin{equation}\label{eq:angles}
\inf_{n\,\in\, \mathbb Z, \,\, 1\,\le\, i\, \le\, d} \,\, \measuredangle (e_n(v_i),F_{n,i}) \,> \,0
\end{equation}
then the basis $\cE$ has $\mathcal S-$bounded projections.
\end{proposition}

\begin{proof}
From the assumption~\eqref{eq:angles}, we deduce that there exists $\theta \in \,\,]0, \pi[$ such that, for every $n \in \mathbb Z$ and $1 \le i \le d$, one has
$$\big|\cos \big(\measuredangle (e_n(v_i),F_{n,i})\big)\big|\,\leq\, \big|\cos\, \theta\big| \, < \, 1.$$
Fix $n \in \N$ and $1 \le i \le d$. For each $x\in X$, write $x = \alpha_i \, e_n(v_i) + f_{n,i}$, where $f_{n,i} \in F_n^i$ is unique. In what follows, assume that $x \neq 0$. 
Using the estimate
\begin{align*}
\|x\|^2   & \,\ge\,  \|\alpha_i\|^2 \, \|e_n(v_i)\|^2 + 
    \|f_{n,i}\|^2 - 2\|e_n(v_i)\|\, \|f_{n,i}\| \,|\cos \theta| \\
    & \,= \,\|\alpha_i\|^2 + \|f_{n,i}\|^2 - 2\|f_{n,i}\|\, |\cos \theta|
    \,>\, 0 
\end{align*}
we obtain
\begin{align}\label{eq:estimate}
\frac{\|\Pi_{e_n(v_i)}(x)\|^2}{\|x\|^2}
    & \,\le\, \frac{\|\alpha_i\|^2}{\|\alpha_i\|^2 + 
    \|f_{n,i}\|^2 - 2\, \|\alpha_i\|\, \|f_{n,i}\|\, |\cos \theta|} \nonumber \\
    &\smallskip \nonumber \\
    & \,\le\, \frac{\|\alpha_i\|^2}{\big(\|\alpha_i\| - \|f_{n,i}\|\big)^2 
    + \big(2 - 2\,|\cos \theta|\big)\, \|\alpha_i\| \, \|f_{n,i}\|}.
\end{align}
If $\alpha_i=0$, the previous quotient is zero and there is nothing to prove. If $\|\alpha_i\|>0$, writing
\begin{eqnarray*}
t&=& \frac{\|f_{n,i}\|}{\|\alpha_i\|} \\ 
\beta &= & 1 - |\cos \theta| \, > \,0
\end{eqnarray*}
then \eqref{eq:estimate} becomes 
\begin{align*}
\frac{\|\Pi_{e_n(v_i)}(x)\|^2}{\|x\|^2}
    & \,\le\, \frac{\|\alpha_i\|^2}{\big(1-t\big)^2 \,\|\alpha_i\|^2 
    + 2 \beta t\, \|\alpha_i\|^2} \\
    & \,\le\, \frac{1}{\big(1-t\big)^2 + 2 \beta t} \\
    & \,\le\,  \frac1{\min_{s\,\geq \,0} \; \big\{\big(s+\beta-1\big)^2 + 2 \beta-\beta^2\big\}} \\
    & \, =\,  \frac1{2 \beta-\beta^2}.
\end{align*}
Thus, $\Pi_{e_n(v_i)}$ is uniformly bounded, independently on $n$ and $i$. This proves the proposition.
\end{proof}

\smallskip


Example \ref{ex:class} shows that, without the assumption \eqref{eq:angles}, the conclusion of Proposition~\ref{prop:angles} may fail.

\subsection{Proof of Corollary~\ref{cor:equiP}}

Let $X$ be a d-dimensional Hilbert space and $\cE $ be an orthogonal basis for $(S_n)_{n \,\in \,\Z}$. In what follows, let property $\mathcal P$ stand for transitivity, mixing, or shadowing. Under the assumptions of Corollary~\ref{cor:equiP}, Theorem~\ref{thm:productfactor} ensures that the shift operator $\sigma_{\mathcal S}$ is conjugate to the product of weighted backward shifts $\prod _{x \,\in\, \cE} B_{\omega(x)} \colon (\ell_p(\mathbb K))^d \to (\ell_p(\mathbb K))^d$. 
Hence, Corollary~\ref{cor:equiP} is a direct consequence of the next lemma.

\begin{lemma}\label{le:productsP}
Let $Y$ be a Banach space and $T_i\colon  Y \to Y$, $1\le i \le d$, be a sequence of invertible linear operators. Consider $T=T_1\times T_2 
\times \dots \times T_d \colon Y^d\to Y^d$, where $Y^d$ is endowed with the sum norm. Then
$$T \text{ has property $\mathcal P$} \quad \Rightarrow \quad \big\{T_i\colon \,1\le i \le d\big\} \,\text{ has equi-property $\mathcal P$}.$$
\end{lemma}

\begin{proof}
Assume that $\mathcal P$ is transitivity or mixing. Given any non-empty open sets $U,V\subset Y$, it suffices to apply the corresponding property for the operator $T$ and the open subsets $U\times U \times \dots \times U$ and $V\times V \times \dots \times V$ of $Y^d$ to obtain the equi-property $\mathcal P$ for the class of operators $\{T_i\colon \,1\le i \le d\}$ acting on $Y$.
\smallskip

Suppose now that $\mathcal P$ is the shadowing property. By definition, there is $K >0$ which witnesses the shadowing property of $T$. We claim that the same $K$ verifies equi-shadowing for $\{T_i\colon \,1\le i \le d\}$. Indeed, for $1 \le i \le d$, for any bounded sequence $\big(z^i_n\big)_{n \,\in\, \Z}$ in $Y$, there is a bounded sequence $\big(x_n\big)_{n \,\in \,\Z}$ in $Y^d$ satisfying 
\[ x_{n+1}-T(x_n) \,= \,(0, \ldots, 0, z^i_n, 0, \ldots, 0 )\quad \quad \text{and} \quad \quad  \sup_{n\, \in\, \Z} \,\|x_n\| < K \sup_{n \,\in\, \Z} \|z_n^i\|.\] 
Letting $x^i_n$ be the $i$th coordinate of $x_n$, we get
\[ x^i_{n+1}-T_i(x^i_n) \,=\,  z^i_n, \quad \quad \text{and} \quad \quad  \sup_{n\, \in\, \Z} \,\|x^i_n\| < K \sup_{n \,\in\, \Z} \,\|z_n^i\|\]
completing the proof of the lemma.
\end{proof}


\section{Open questions}\label{se:open}
We showed in Corollary~\ref{cor:GenHypShadowing} that, under an extra assumption, the answer to the following questions is yes. However, the general problem remains open.
\begin{question}
For a general shift operator, is the shadowing property equivalent to generalized hyperbolicity?
\end{question}

Example~\ref{discontinuousprojections} shows that there are shift operators which are not linearly conjugate to a finite product of particular types of weighted shifts. Yet, the following remains open.
\begin{question}
Are there shift operators, determined by sequences of matrices on finite dimensional Hilbert spaces, such that $\sigma_S$ is not linearly conjugate to a finite product of shifts?
\end{question}

Weighted backward shifts on $\ell_p(\R)$ does not distinguish Devaney chaos from frequent hypercyclicity \cite{BayarRuzsa2015,CharpentierEtAl2023}. It would be interesting to have a negative solution to the following question.
\begin{question}
Do the notions of Devaney chaos and frequent hypercyclicity coincide for shift operators on $\ell_p(X)$? The answer is yes when $X = \R$.
    
\end{question}

\bibliographystyle{siam}
\bibliography{biblio}

\bigskip
\bigskip

\noindent Maria Carvalho\\ 
CMUP \& Departamento de Matem\'atica,\\ 
Faculdade de Ci\^encias da Universidade do Porto, \\
Rua do Campo Alegre s/n, 4169-007 Porto, Portugal.\\
E-mail: mpcarval@fc.up.pt\\

\noindent Udayan B. Darji\\ 
Department of Mathematics,\\ 
University of Louisville,\\
Louisville KY 40292, USA.\\
E-mail: ubdarj01@louisville.edu\\

\noindent Paulo Varandas\\ CMUP \& Faculdade de Ci\^encias da Universidade do Porto, \\
Rua do Campo Alegre s/n, 4169-007 Porto, Portugal.\\ 
Departamento de Matem\'atica e Estat\'istica,\\
Universidade Federal da Bahia,\\
Av. Ademar de Barros s/n, 40170-110 Salvador, Brazil.\\
E-mail: paulo.varandas@ufba.br

\end{document}